\documentclass[preprint,amssymb]{elsarticle}
\usepackage{graphicx, epsfig, amssymb} %include figure files
\usepackage{amsmath, amsfonts, amsthm}
\usepackage{bm} %include bold math: \bm{} creates bold letters in math mode
\usepackage{pstricks,psfrag,pst-node,pst-text,pst-3d,pst-grad}

\usepackage[margin=2.5cm]{geometry}
\usepackage[linktocpage]{hyperref}
\usepackage[caption=false]{subfig}
\usepackage{diagbox}

\newcommand{\dif}{{\rm d}}
\newcommand{\del}{\partial}

\newcommand{\dx}{\delta x}

\newtheorem{thm}{Theorem}
\newtheorem{lem}{Lemma}
\newtheorem{prop}{Proposition}

\newtheorem{defn}{Definition}

\begin{document}
\title{The W4 method: a new multi-dimensional root-finding scheme for nonlinear systems of equations}

\author{Hirotada Okawa}\ead{h.okawa@aoni.waseda.jp}
\address{Yukawa Institute for Theoretical Physics, Kyoto University, Kyoto 606-8502, Japan}
\address{Research Institute for Science and Engineering, Waseda University, Tokyo 169-8555, Japan}
\author{Kotaro Fujisawa}
\address{Research Institute for Science and Engineering, Waseda University, Tokyo 169-8555, Japan}
\author{Yu Yamamoto}
\address{Research Institute for Science and Engineering, Waseda University, Tokyo 169-8555, Japan}
\author{Ryosuke Hirai}
\address{Department of Physics, University of Oxford, Keble Rd, Oxford, OX1 3RH, United Kingdom}
\author{Nobutoshi Yasutake}
\address{Physics Department, Chiba Institute of Technology, Chiba 275-0023, Japan}
\address{Advanced Science Research Center, Japan Atomic Energy Agency, Tokai, Ibaraki 319-1195, Japan}
\author{Hiroki Nagakura}
\address{TAPIR, Walter Burke Institute for Theoretical Physics, Mailcode 350-17, California Institute for Technology, Pasadena, CA 91125, USA}
\address{Department of Astrophysical Sciences, Princeton University, Princeton, NJ 08540, USA}
\author{Shoichi Yamada}
\address{Research Institute for Science and Engineering, Waseda University, Tokyo 169-8555, Japan.}
\address{Science and Engineering, Waseda University, Tokyo 169-8555, Japan.}

\date{\today} 

\begin{abstract} 
 We propose a new class of method for solving nonlinear systems of equations,
 which, among other things,
 has four nice features:
 (i) it is inspired by the mathematical property of damped oscillators,
 (ii) it can be regarded as a simple extention to the Newton-Raphson(NR) method,
 (iii) it has the same local convergence as the NR method does,
 (iv) it has a significantly wider convergence region or the global convergence than that of the NR method.
 In this article, we present the evidence of these properties,
 applying our new method to some examples and comparing it with the NR method.
\end{abstract}

% \pacs{
% 02.60.Cb,% Numerical simulation; solution of equations 
% 02.70.−,% Computational techniques; simulations 
% 02.90.+p,% Other topics in mathematical methods in physics
% }
\maketitle

%%%%%%%%%%%%%%%%%%%%%%%%%%%%%%%%%%%%%%%%%%%%%%%%%%%%%%%%%%%%%%%%%%%%%%%%%%%%%%%%%
\section{Introduction}\label{sec:intro}
%%%%%%%%%%%%%%%%%%%%%%%%%%%%%%%%%%%%%%%%%%%%%%%%%%%%%%%%%%%%%%%%%%%%%%%%%%%%%%%%%
%
One of the most important problems in computational science and
engineering is a root-finding of functions.
It is defined as a problem to solve the following equations numerically:
\begin{eqnarray}
 \bm{F}(\bm{x}) = 0,\label{eq:nonlinearEQs}
\end{eqnarray}
where $\bm{x} \in \mathbb{R}^{N}$, $N\in \mathbb{Z}$ and $\bm{F}: \mathbb{R}^{N}\rightarrow\mathbb{R}^{N}$
is a generic system of nonlinear equations in $\bm{x}$.

Numerical root-finding methods are essentially categorized into two types. 
One is deterministic
methods such as the Newton-Raphson(NR) method~\cite{kelley2003},
which always give the same answer to the same initial condition.
The other is stochastic methods such as the Monte-Carlo method~\cite{jablonski1980},
in which the answer may change in each trial even for idential initial conditions.
The former method normally requires a larger computational cost
but needs a very small number (mostly one) of trials,
whereas the computational cost per trial is small but many trials are
needed in the stochastic method.
 
Although the single-variable problem is rather simple,
it becomes highly nontrivial
when the dimension of variable space is higher than one.
One may choose the NR method
if the dimension is not prohibitively large in terms of numerical cost.
The NR method can be regarded as an iterative solver based on the fixed point theorem
and it is guaranteed to give a solution
as long as the initial guess is sufficiently close to the solution\cite{ortega1970}.
In practice, there are two major drawbacks in the NR method:
(i) very heavy computational cost in the inversion of a large Jacobian matrix
(See Eq.~\eqref{eq:NRmap})
and (ii) the strong dependence of convergence on the initial guess.
The former is evaluated as $\mathcal{O}\left(N^3\right)$
for the direct inversion of an $N\times N$ Jacobian matrix
in the system of $N$-dimensions.
So far, much attention has been paid to issue~(i)
and many efforts have been made to reduce the cost: for example, a
number of quasi-Newton methods have been proposed, which successfully
reduce the computational costs to $\mathcal{O}\left(N^2\right)$\cite{broyden1965,kelley2003}.
As for issue~(ii),
many studies have attempted to accelerate
the normally quadratic convergence of the NR method.
In fact, Halley's and Householder's methods proposed for single-variable
problems are headed in this direction~\cite{householder1970}.
Also for multi-variable problems, a new approach emerged recently~\cite{ramos2015,ramos2017}.

It is the sensitivity of convergence to the initial guess that we focus most in this paper.
It is well known that the NR method fails to find solutions when the initial
guess is far from the solution, in particular, in multi-dimensional problems.
We need to give a good consideration to the initial guess but
there is no guarantee that it is always successful unless we know the solution.
When the NR method fails to find solutions,
we see either a divergence or a permanent oscillation of the successive
solutions of the linearized equations as explicitly shown in Sec.~\ref{sec:problem}.
This is sometimes related to the existence of singularities in the
Jacobian matrix and may be solved by modifying source terms in the
linearized equations~\cite{kou2006,wu2007,hueso2009}.
It should also be emphasized that the inversion of ill-conditioned
matrices usually need high-precision calculations regardless of the
iterative methods employed
and a careless handling could easily produce an incorrect inverse matrix
whose component is wrong from the leading figure.
In such cases, we should pay proper attention to numerical errors accumulated during iterations.
In this article, we focus especially on problems, for which the NR
method fails to find solutions
in spite of the fact that its Jacobian matirix is well-conditioned initially.
The initial guess is likely to be a culprit then. We demonstrate that
our new method can find a solution even when an initial guess is far
from the solution.\\

%%% Contents
This paper is organized as follows.
Firstly, we define the problem we deal with in this paper.
It will also make clear the motivation of this work, employing some
concrete problems, in Sec.~\ref{sec:problem}.
Then, we present our new formulation to find roots of a system of nonlinear equations in Sec.~\ref{sec:method},
starting with the descriptions of other conventional methods in the literature on the same basis.
In Sec.~\ref{sec:result}, we demonstrate the capabilities of our method for several sample problems,
comparing them with those of other methods.
Finally, we give a summary and comments on future prospects
in Sec.~\ref{sec:conclusion}.

%%%%%%%%%%%%%%%%%%%%%%%%%%%%%%%%%%%%%%%%%%%%%%%%%%%%%%%%%%%%%%%%%%%%%%%%%%%%%%%%%
\section{Problem}\label{sec:problem}
%%%%%%%%%%%%%%%%%%%%%%%%%%%%%%%%%%%%%%%%%%%%%%%%%%%%%%%%%%%%%%%%%%%%%%%%%%%%%%%%%
%
As mentioned in the Introduction, our goal is to numerically solve generic
systems of nonlinear equations, Eq.~\eqref{eq:nonlinearEQs}.
In order to make clear the issues
of our concern in this article, we consider some specific examples in
the following.
Although we are interested in multi-variable equations,
we start with a single-variable problem for clarity.

\subsection{single-variable problem}
We first consider the root-finding of the following single-variable equation:
\begin{eqnarray}
 f(x) = \arctan (x) +\sin (x) -1 = 0,\label{eq:simple1d}
\end{eqnarray}
where $x\in \mathbb{R}$.
In the left panel of Fig.~\ref{fig:ex0_simple1d}, we display the
function, which has oscillatory behavior and has multiple solutions only
on the positive $x$-axis($x>0$).
In the NR method, one employs the following iteration map from the $n$-th step to $(n+1)$-th step,
\begin{eqnarray}
 {x}_{n+1} = {x}_{n} - \left\{\frac{\dif f}{\dif x}(x_n)\right\}^{-1}{f}({x}_{n}).\label{eq:NRmap_simple1d}
\end{eqnarray}
As demonstrated in Table~\ref{tab:comp_ex0_simple1d},
as long as it obtains a solution, the
NR method requires only a small number of iteration steps to get to the
solution within a prescribed accuracy, i.e., $\left|f(x)\right|<10^{-6}$
in this example,
 thanks to its well known quadratic-convergence nature.
The table also demonstrates that the NR method fails to find a solution
if the initial guess is smaller than $x_i=-2.0$.
We show typical trajectories of the NR map,
Eq.~\eqref{eq:NRmap_simple1d}, both for the successful and 
failed attempts in the right panel of Fig.~\ref{fig:ex0_simple1d}.
The iteration rapidly converges to one of the solutions$(x=1)$
when the initial guess is
close to the solution whereas
it shows an oscillation between two values and never reaches a solution
if the initial guess is $x_i=-2.0$.
Since the latter oscillation is a consequence of the oscillatory nature
of the function~$f(x)$,
it is natural to soften its slope somewhat by hand in the evaluation of
$x$ in the next step,
\begin{eqnarray}
 {x}_{n+1} = {x}_{n} - \Delta\tau\left\{\frac{\dif f}{\dif x}(x_n)\right\}^{-1}{f}({x}_{n}),\label{eq:DNmap_simple1d}
\end{eqnarray}
where we introduce a damping parameter $\Delta\tau$,
which is a positive real number less than unity.
\begin{figure}[t]
 \begin{tabular}{cc}
  \includegraphics[width=8.4cm,clip]{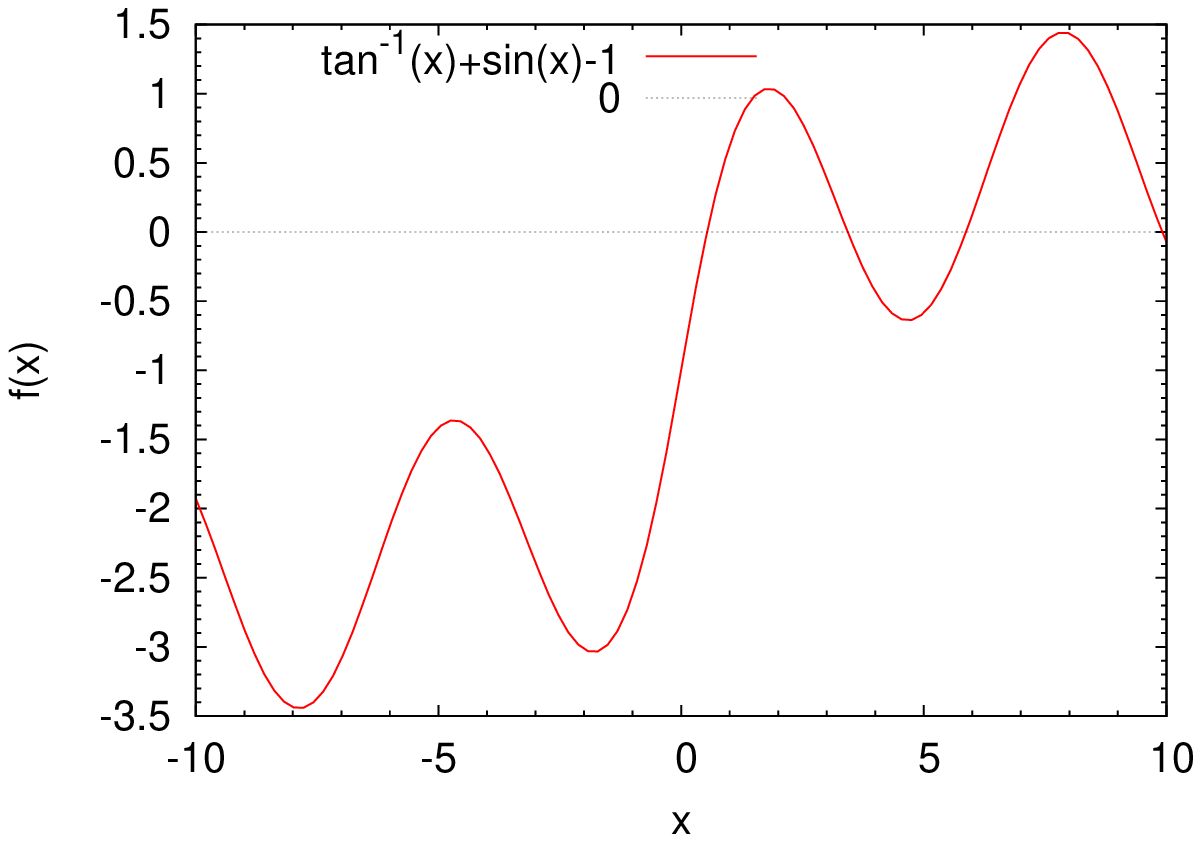}
  & 
  \includegraphics[width=8.4cm,clip]{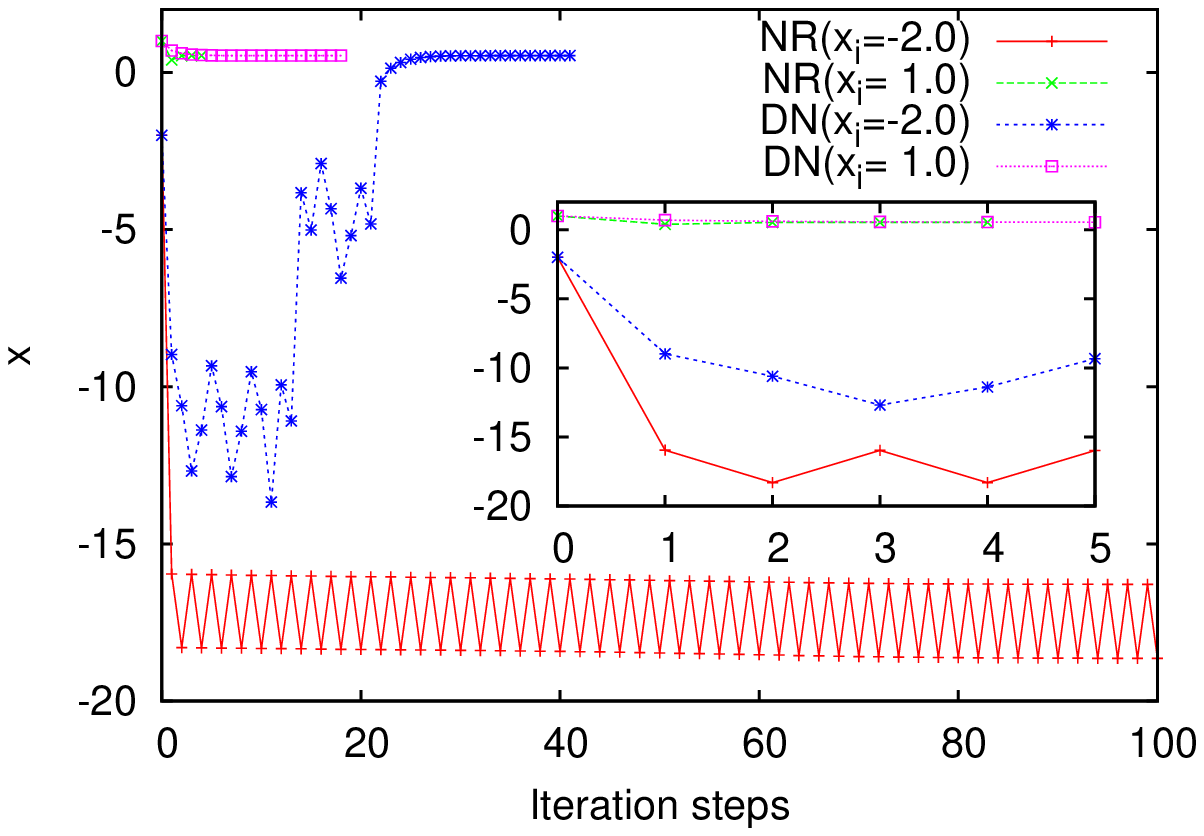}
 \end{tabular}
 \caption{(Left) The function given in~\eqref{eq:simple1d}.
 (Right) Some sequences produced by the iteration maps for
 the NR(blue and purple)
 and DN(orange and green) methods.
 The initial conditions are $x_i=1.0$ and $x_i=-2.0$.
 }
 \label{fig:ex0_simple1d}
\end{figure}
The results for this so-called Damped Newton(DN)
method are presented in Table~\ref{tab:comp_ex0_simple1d}
for the comparison with the original NR method.
Although the number of iteration steps required
to get a solution within the prescribed accuracy increases
substantially in the DN method,
the region of the initial guess where a solution is obtained becomes wider.
This implies that the NR method, if appropriately modified, may find a
solution even for initial guesses that are not very close to the solution
at least for single-variable problems in principle.
\begin{table}[t]
 \begin{center}
  \begin{tabular}{c||c|c|c|c|c|c|c|c|c|c|c|c|c|}
   \diagbox{Method}{$x_i$} & -3.0& -2.5& -2.0& -1.5& -1.0& -0.5& 0& 0.5& 1.0& 1.5& 2.0& 2.5& 3.0\\\hline
   NR & $\infty$& $\infty$& $\infty$& 4& 5& 4& 3& 2& 4& 8& 4& 4& 3\\
   DN & 25& $\infty$& 41& 20& 19& 20& 19& 15& 18& 19& 17& 19& 18\\
  \end{tabular}
 \end{center}
 \caption{Necessary iteration steps to converge by iterative methods in
 the literature from several initial guesses~$x_i$.
 NR method stands for Newton-Raphson method and DN stands for Damped Newton method.
 The damping parameter for the DN method is set as $\Delta\tau=0.5$.
 Stopping criterion in both methods is $|f(x)|<10^{-6}$.
 $\infty$ denotes that no solution is obtained by $10^4$ iteration steps.
 }
 \label{tab:comp_ex0_simple1d}
\end{table}

\subsection{multi-variable problem}
The above situation is changed drastically when the dimension of variable space becomes higher than one.
To observe what happens, we consider the following concrete problem:
\begin{subequations}
 \label{eq:simple2d}
 \begin{align}
  F_1(x,y) &:= x^2 +y^2 -4 = 0,\\
  F_2(x,y) &:= x^2y -1 = 0.
 \end{align}
\end{subequations}
The functions~$F_1$ and $F_2$ represent a circle and a parabola, respectively,
and the roots of this system correspond to the intersections of
these two curves in the $x-y$ plane.
In the NR or DN method for multi-variable problems~\cite{ortega1970,kelley2003},
one can proceed just as in the single-variable case and
employs the following iteration map from the $n$-th step to the
$(n+1)$-th step,
\begin{eqnarray}
 \bm{x}_{n+1} = \bm{x}_{n} - \Delta\tau J^{-1}\bm{F}(\bm{x}_{n}),\label{eq:NRmap}
\end{eqnarray}
where $\bm{F}:= \left(F_1\ F_2\right)^T$ and $J^{-1}$ is the inverse
of the Jacobian matrix associated with Eqs.~\eqref{eq:simple2d},
which is given as
\begin{eqnarray}
 J =
 \begin{bmatrix}
  2x & 2y\\
  2xy & x^2
 \end{bmatrix}.\label{eq:simple2d_Jacobian}
\end{eqnarray}
Then, the iteration maps for $\bm{x}=\left(x\ y\right)^T$ are
expressed explicitly as
\begin{subequations}
 \label{eq:NRmap_simple2d}
 \begin{align}
  x_{n+1} &= x_{n} -\Delta\tau\left[\frac{x_{n}}{2\left(x_{n}^2-2y_{n}^2\right)}F_1 -\frac{y_{n}}{x_{n}\left(x_{n}^2-2y_{n}^2\right)}F_2\right], \\
  y_{n+1} &= y_{n} -\Delta\tau\left[-\frac{y_{n}}{x_{n}^2-2y_{n}^2}F_1 +\frac{1}{x_{n}^2-2y_{n}^2}F_2\right].
 \end{align}
\end{subequations}
\begin{figure}[t]
 \begin{tabular}{cc}
  \includegraphics[width=8.cm,clip]{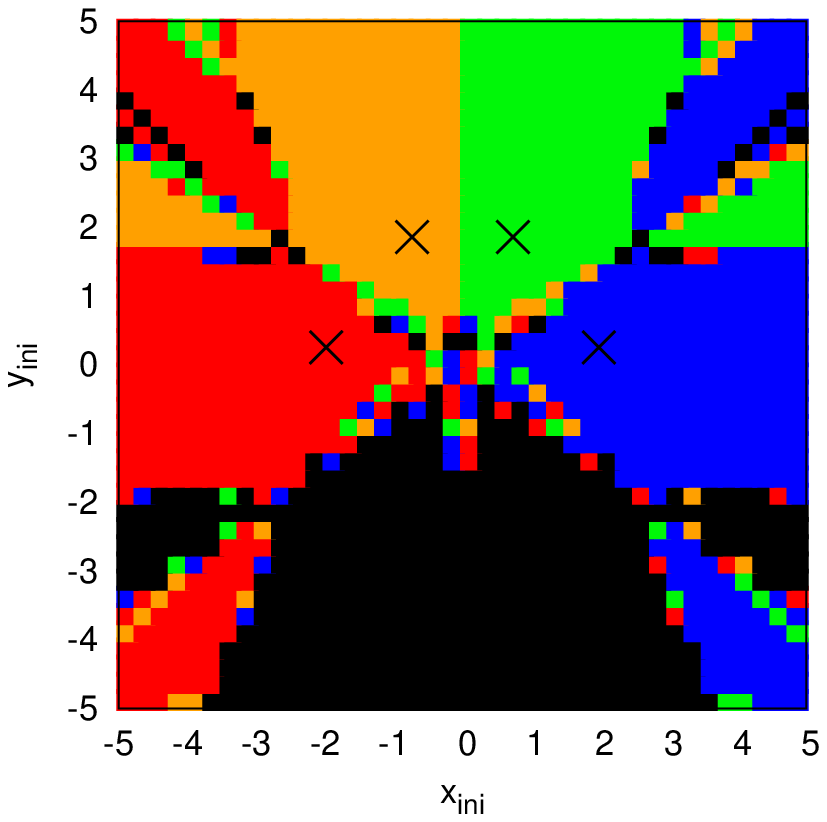}&
  \includegraphics[width=8.cm,clip]{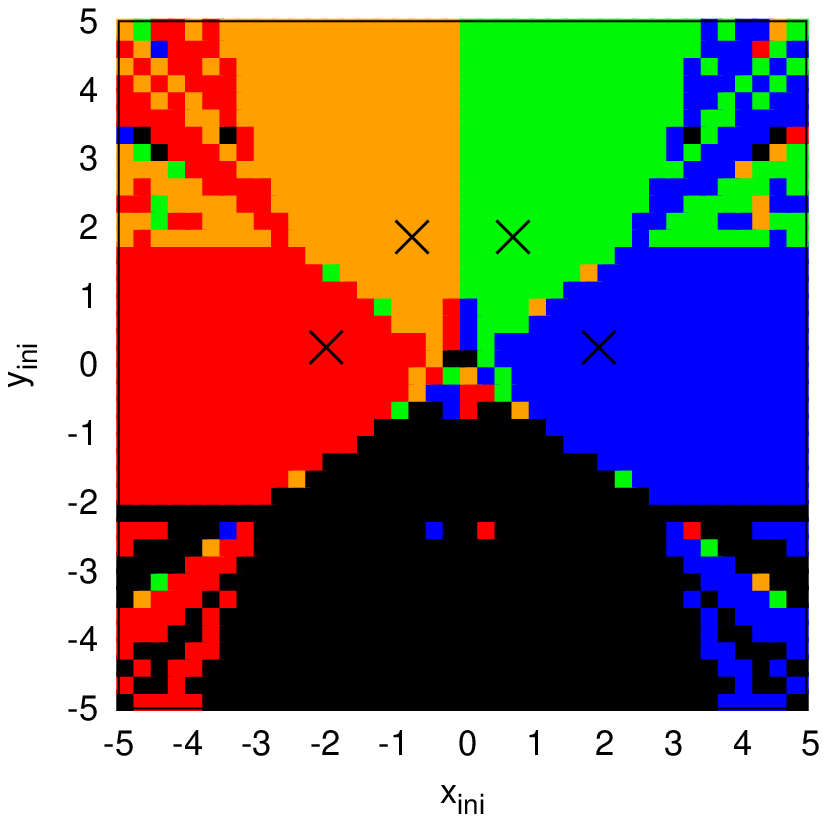}
 \end{tabular}
 \caption{Initial guess dependence by the NR method(Left) and the DN method(right))
 with $\Delta\tau=0.5$(Right).
 The red, orange, green and blue correspond to all solutions of Eqs.~\eqref{eq:simple2d},
 while the black means the method cannot find any solution.
 In addition, four solutions are expressed by the black crosses.
 }
 \label{fig:ex0_simple2d_NR}
\end{figure}

In Fig.~\ref{fig:ex0_simple2d_NR}, we display the Newton basin for the NR and DN maps
in the left and right panels, respectively,
as an indicator of the initial-guess dependence.
Different colors(red, orange, green, blue) correspond to one of the 
four solutions,
\begin{align}
 (x^{*},y^{*}) \sim
 (\pm 1.9837924, 0.25410169),\quad
 (\pm 0.73307679, 1.8608059),\label{eq:simple2d_Solution}
\end{align}
to which each point converges.
The black color implies that none of the solutions has been reached within $1000$ iteration steps.
For both methods, the local convergence is observed more or less as expected.
It is evident, however, that the damping factor introduced in the DN
method does not help as
much in finding solutions in this two dimensional problem
compared with the one-dimensional counterpart considered in the previous section.\\

%%%%%%%%%%%%%%%%%%%%%%%%%%%%%%%%%%%%%%%%%%%%%%%%%%%%%%%%%%%%%%%%%%%%%%%%%%%%%%%%%
\section{The New Method}\label{sec:method}
%%%%%%%%%%%%%%%%%%%%%%%%%%%%%%%%%%%%%%%%%%%%%%%%%%%%%%%%%%%%%%%%%%%%%%%%%%%%%%%%%
%

\subsection{The NR method as a relaxation}\label{ssec:NR}
The iteration map given in Eq.~\eqref{eq:NRmap} can also be regarded
as a discretized form of the continuous time-evolution equation:
\begin{eqnarray}
 \frac{\dif \bm{x}}{\dif\tau} =-J^{-1}\bm{F}(\bm{x}).\label{eq:NRbase0}
\end{eqnarray}
We will consider a bit more general time-evolution equation with a
preconditioning matrix~$M$ given as
\begin{eqnarray}
 \frac{\dif \bm{x}}{\dif\tau} =-M\bm{F}(\bm{x}).\label{eq:NRbase1}
\end{eqnarray}
The corresponding discretized form is obtained by the Euler forward
discretization in time
with the time interval~$\Delta \tau$ as
\begin{eqnarray}
 \bm{x}_{n+1} = \bm{x}_{n} - \Delta\tau M\bm{F}(\bm{x}_{n}),\label{eq:NRmap_M}
\end{eqnarray}
where $\bm{x}_{n}$ is the discretized solution at the $n$-th iteration step.
If one employs the inverse of the Jacobian matrix as the preconditioner~$M=J^{-1}$,
then Eq.~\eqref{eq:NRmap} is recovered.

\begin{thm}\label{thm:NRmap_M}
 Suppose the solution of Eq.~\eqref{eq:nonlinearEQs} exists
 and let $\bm{F}:\mathbb{R}^{n}\rightarrow \mathbb{R}^{n}$ be Lipschitz
 continuous near the solution~$\bm{x}^{*}$,
 then the iteration map~\eqref{eq:NRmap_M}
 for~$M=J^{-1}$ with $0<\Delta\tau<1$ converges linearly to the
 solution~$\bm{x}^{*}$
 starting from an arbitrary initial guess sufficiently close to the solution.
\end{thm}

\begin{proof}
 The Taylor expansion of the source function~$\bm{F}(\bm{x})$ around the solution~$\bm{x}^{*}$
 up to the first order in the error defined as $\bm{e}_n := \bm{x}^{*} -\bm{x}_n$ gives
 \begin{eqnarray}
  \bm{F}\left(\bm{x}^{*}\right)
   = \bm{F}\left(\bm{x}_{n}\right) +J \bm{e}_n +\mathcal{O}(||\bm{e}||^2).
 \end{eqnarray}
 Since $\bm{x}^{*}$ is the solution and $\bm{F}(\bm{x}^{*})=0$,
 one can express $\bm{F}(\bm{x}_n)$ in Eq.~\eqref{eq:NRmap_M}
 in terms of the Jacobian matrix and the error to yield
 the error propagation equation as
 \begin{eqnarray}
  \bm{e}_{n+1} = \left( I - \Delta\tau MJ\right) \bm{e}_{n},
 \end{eqnarray}
 where $I$ is the $N\times N$ identity matrix.
 Under the present condition~$M=J^{-1}$ and
 $0<\Delta\tau<1$ the error converges to zero in  the~$n\rightarrow\infty$
 limit\footnote{In general, the necessary and sufficient conditions for
 the map~\eqref{eq:NRmap_M} to converge to a solution are that the
 absolute values of all the eigenvalues for the matrix~$E - \Delta\tau
 MJ$ should be less than unity.},
 \begin{eqnarray}
  \lim_{n\rightarrow\infty}\frac{||\bm{e}_{n+1}||}{||\bm{e}_{0}||} =
   \lim_{n\rightarrow\infty}\left(1-\Delta\tau\right)^{n+1} = 0.
 \end{eqnarray}
 As a result, we obtain the solution~$\bm{x}^{*}$ by the iteration map~\eqref{eq:NRmap_M}.
\end{proof}
 We again emphasize that
this theorem guarantees only the local convergence,
i.e. a solution is obtained only when the initial guess is close enough
to the solution.

\subsection{A new approach for single-variable problems}\label{ssec:W41D}
Now we proceed to the new approach
we propose in this article.
For a nonlinear equation with a single variable
\begin{eqnarray}
 F(x)=0,
  \label{eq:single1}
\end{eqnarray}
we consider the following time-evolution equation with
the second derivative in time
\begin{eqnarray}
 \frac{\dif^{2} x}{\dif \tau^{2}} +c\frac{\dif x}{\dif \tau}
  +\left\{\frac{{\rm d}F}{{\rm d}x}\right\}^{-1} F(x) =0,
\label{eq:W4equation0}
\end{eqnarray}
or equivalently two first-order equations,
\begin{subequations}
 \label{eq:W4equation1}
 \begin{align}
  \frac{\dif x}{\dif\tau} := \dot{x} &= p,\\
  \frac{\dif p}{\dif\tau} := \dot{p} &= -cp -\left\{\frac{{\rm d}F}{{\rm d}x} \right\}^{-1} F(x),
 \end{align}
\end{subequations}
where the dot notation is introduced for the derivative with respect to $\tau$ for
later convenience and $c\in \mathbb{R}$ is a constant parameter.
These are nothing but damped oscillator's equations of motion if $F(x)=x$
and $c>0$.

\begin{lem}\label{lem:decaymode}
 Suppose there exists a solution of the nonlinear equation~\eqref{eq:single1}
 indicated with $x^{*}$,
 if one sets the parameter~$c$=2,
 then a sufficiently small deviation from $x^{*}$
 always decays during its time evolution~\eqref{eq:W4equation0}.
\end{lem}

\begin{proof}
 The Taylor expansion around the solution yields
 \begin{eqnarray}
  F(x^{*}+\dx) = F(x^{*}) +\frac{\dif F}{\dif x}\dx
   +\mathcal{O}(\dx^2).
 \end{eqnarray}
 One can then obtain the equation for the deviation~$\dx$
 valid up to the first order
 by the substitution~$x=x^{*}+\dx$ into Eq.~\eqref{eq:W4equation0}
 as
 \begin{eqnarray}
  \ddot{\dx} +c\dot{\dx} +\dx = 0.\label{eq:perturbation0}
 \end{eqnarray}
 There exists a characteristic mode in the
 equation~\eqref{eq:perturbation0}, which is
 explicitly given by assuming $\dx = e^{i\omega\tau}$ as
 \begin{eqnarray}
  \omega = \frac{ic \pm\sqrt{-c^2 +4}}{2}.
 \end{eqnarray} 
 Thus, the deviation from the solution~$x^{*}$
 with $c=2$ decays exponentially as $\dx=e^{-\tau}$ during the time evolution.
\end{proof}

\begin{thm}(Bendixson's criterion)\label{thm:bendixson}
 If on a simply connected region $\mathcal{D}\subset R^2$ in the phase space~$(x,p)$
 the parameter~$c$ is nonzero, then there is no periodic orbits in $\mathcal{D}$.
\end{thm}

\begin{proof}
 Any orbit of system~\eqref{eq:W4equation1} in the two-dimensional real phase space~$(x,p)$ is described by
 $\dif p/\dif x = \dot{p}/\dot{x}$.
 A closed orbit~$\Gamma$ in the phase space satisfies
 \begin{align}
  \oint_{\Gamma} \Bigl( \dot{x}\dif p -\dot{p}\dif x \Bigr)= 0.
 \end{align}
 However, for a closed area~$\mathcal{D}$ bounded by $\Gamma$, 
 Green's theorem gives 
 \begin{align}
  \oint_{\Gamma}\Bigl( \dot{x}\dif p -\dot{p}\dif x \Bigr)
  = \int\!\!\!\!\int_{\mathcal{D}}\left(\frac{\del \dot{x}}{\del x}
  +\frac{\del \dot{p}}{\del p}
  \right)\dif x\dif p = -c \mathcal{D}.
 \end{align}
 Consequently, no periodic orbit exists inside $\mathcal{D}$ in the phase space
 with nonzero parameter~$c$.
\end{proof}

Lemma~\ref{lem:decaymode} combined with Theorem~\ref{thm:bendixson} implies that 
the solution of our new time-evolution equation~\eqref{eq:W4equation0}
is all settled down asymptotically to one of the solutions of Eq.~\eqref{eq:single1}
\footnote{This analysis will not hold for higher
dimensional problems since chaos may kick in and trajectories will not
be closed.}.\\

We now show that the discretized version of
equations~\eqref{eq:W4equation1}:
\begin{subequations}
 \label{eq:W4map_1}
 \begin{align}
  x_{n+1} &= x_{n} + \Delta\tau\ p_{n},\label{eq:W4map_1x}\\
  p_{n+1} &= (1-c\Delta\tau) p_{n}
  -\Delta\tau\left\{\frac{{\rm d}F}{{\rm d}x} \right\}^{-1} F(x_{n}),\label{eq:W4map_1p}
 \end{align}
\end{subequations}
also converges to the same solution under a certain condition.
In deriving these equations, we
employ the explicit Euler difference in time with an interval~$\Delta\tau$.

\begin{prop}
 Suppose there exists a solution of Eq.~\eqref{eq:single1}
 and let $F:\mathbb{R}\rightarrow \mathbb{R}$ be Lipschitz
 continuous near the solution~$x^{*}$,
 then the iteration map~\eqref{eq:W4map_1} with~$c=2$
 and $0<\Delta\tau<1$ yields a series of~$x_n$ that converges to the
 solution~$x^{*}$
 starting from $p_0=0$ and an arbitrary initial guess~$x_0$ that is
 sufficiently close to the solution~$x=x^{*}$.
\end{prop}

\begin{proof}
Just as in the proof of Theorem~\ref{thm:NRmap_M},
we evaluate the error propagation for the map~\eqref{eq:W4map_1}
 up to the first order in $e_{n}^{(x)}:=x^{*}-x_{n}$ and
 $e_{n}^{(p)}:=0-p_{n}$,
 using $F(x_{n})= \frac{{\rm d}F}{{\rm d}x} e_{n}^{(x)}$, as follows:
 \begin{subequations}
  \label{eq:W4map1_linear}
   \begin{align}
    e_{n+1}^{(x)} &= e_{n}^{(x)} -\Delta\tau e_{n}^{(p)},\label{eq:W4map1_linear_x}\\
    e_{n+1}^{(p)} &= (1-2\Delta\tau)e_{n}^{(p)} +\Delta\tau e_{n}^{(x)}.\label{eq:W4map1_linear_p}
   \end{align}
 \end{subequations}
 We then obtain the following solution to these equations
 \begin{subequations}
  \label{eq:sol_W4map1_linear}
  \begin{align}
   e^{(x)}_n &= e^{(x)}_0\left(1-\Delta\tau\right)^{n-1}\left[1+\left(n-1\right)\Delta\tau\right],\\
   e^{(p)}_n &= ne^{(x)}_0\Delta\tau \left(1-\Delta\tau\right)^{n-1},
  \end{align}
 \end{subequations}
which can be easily confirmed by substitution.
The convergence is now evident for~$0<\Delta\tau<1$,
 \begin{eqnarray}
  \lim_{n\rightarrow\infty}e^{(x)}_{n}=
   \lim_{n\rightarrow\infty}e^{(p)}_{n}=0.
 \end{eqnarray}
\end{proof}

\subsection{The new formulation for multi-variable problems}\label{ssec:W4}
Now we move on to the multi-variable case,
the main part of this paper.
Unlike the NR method discussed in Sec.~\ref{ssec:NR},
the following second-order time-evolution equations
instead of the first-order ones, Eq.~\eqref{eq:NRbase1}:
\begin{eqnarray}
 \frac{\dif^2 \bm{x}}{\dif \tau^2}
  +M_1\frac{\dif \bm{x}}{\dif\tau} +M_2\bm{F}(\bm{x}) = 0,\label{eq:W4base}
\end{eqnarray}
where we introduce two preconditioning matrices, $M_1, M_2: \mathbb{R}^{N}\rightarrow \mathbb{R}^{N}$.
These equations are then divided into two sets of first-order
differential equations as follows:
\begin{subequations}
 \label{eq:W4equation}
 \begin{align}
  \frac{\dif \bm{x}}{\dif\tau} &= X\bm{p},\\
  \frac{\dif \bm{p}}{\dif\tau} &= -2\bm{p} -Y\bm{F},
 \end{align}
\end{subequations}
which $\bm{x}\in \mathbb{R}^{N}$ and $\bm{p}\in \mathbb{R}^{N}$ are
generalized coordinates and associated momenta, respectively;
another pair of preconditioners~$X$ and $Y$ are introduced,
which will be determined later.
We also adopt the damping factor~$c=2$ for the multi-variable case.
The explicit Euler discretization in time with the time interval~$\Delta \tau$
then yields the following iteration map:
\begin{subequations}
 \label{eq:W4map}
 \begin{align}
  \bm{x}_{n+1} &= \bm{x}_{n} +\Delta\tau X\bm{p}_{n},\\
  \bm{p}_{n+1} &= \left(1-2\Delta\tau\right)\bm{p}_{n} -\Delta\tau Y\bm{F}(\bm{x}_{n}).
 \end{align}
\end{subequations}

Again linearizing the above nonlinear map around
the solution~($\bm{x}=\bm{x}^{*}$ and $\bm{p}=\bm{0})$,
we obtain the following error propagation equations:
\begin{subequations}
 \label{eq:W4map_linear}
  \begin{align}
   \bm{e}^{(x)}_{n+1} =& \bm{e}^{(x)}_{n} - X \Delta\tau \bm{e}^{(p)}_{n},\\
   \bm{e}^{(p)}_{n+1} =& \left(1 -2\Delta\tau\right)\bm{e}^{(p)}_{n}
   +YJ \Delta \tau\bm{e}^{(x)}_{n},
  \end{align}
\end{subequations}
which can be cast into a more compact form as
\begin{eqnarray}
 \bm{e}^{(z)}_{n+1} = W \bm{e}^{(z)}_{n},\quad
 \bm{e}^{(z)}_{n} :=
 \begin{pmatrix}
  \bm{e}^{(x)}_{n}\\
  \bm{e}^{(p)}_{n}
 \end{pmatrix}
 ,\quad
 W :=
 \begin{bmatrix}
  I & -\Delta\tau X\\
  \Delta\tau Y J & (1-2\Delta\tau)I
 \end{bmatrix}
 ,\label{eq:W4_matrix}
\end{eqnarray}
where $I$ denotes the $N\times N$ identity matrix.
We name this map the W4 map.

\begin{lem}\label{lem:error}
 Suppose there exists a complete set of eigenvectors~$\bm{v}_i\in
 \mathbb{R}^{N}$ of the matrix~$W$ and 
 let $P$ be an $N\times N$ matrix composed of~$\bm{v}_i$ and $\lambda_i$
 be the eigenvalues corresponding to $\bm{v}_{i}$,
 then the error~$\bm{e}^{(z)}_n$ converges to the zero vector if the maximum eigenvalue satisfies $|\lambda_{max}|<1$.
\end{lem}

\begin{proof}
 The matrix~$W$ can be decomposed as $W=P^{-1}\Lambda P$ in terms of the
 matrix~$P:=\left[\bm{v}_{1}\ \bm{v}_{2}\ \cdots\ \bm{v}_{N}\right]$ and the
 diagonal matrix~$\Lambda:={\rm diag}\left[\lambda_1,\lambda_2,\cdots,\lambda_N\right]$.
 Then the error propagates as
 \begin{eqnarray}
  \bm{e}_{n+1} = P^{-1}\Lambda P\bm{e}_{n}
   = \left(P^{-1}\Lambda P\right)^{n+1}\bm{e}_{0}
   = P^{-1}\Lambda^{n+1} P\bm{e}_{0}
 \end{eqnarray}
 and converges to the zero vector if the maximum eigenvalue~$|\lambda_{max}|$ is less than unity. 
\end{proof}

Our problem is then reduced to finding the eigenvalues of the
matrix~$W$ in Eq.~\eqref{eq:W4_matrix}.

\begin{lem}\label{lem:blockmatrix}
 Suppose $A,B,C,D$ are $N\times N$ matrices
 and $M:= \begin{bmatrix} A & B\\ C & D \end{bmatrix}$ is a $2N\times 2N$ matrix.
 If $A$ and $B$ commute with each other and $A$ is invertible,
 then the determinant of the matrix~$M$ is given as
 \begin{eqnarray}
  \det \left[M\right]
   = \det\left[AD -CB\right].
   \quad
 \end{eqnarray}
\end{lem}

\begin{proof}
 Following Ref.~\cite{silvester2000},
 we consider the identity
 \begin{eqnarray}
   \begin{bmatrix}
    A & B\\
    C & D
   \end{bmatrix}
   \begin{bmatrix}
    I & -B\\
    O & A
   \end{bmatrix}
   =
   \begin{bmatrix}
    A & BA -AB\\
    C & DA -CB
   \end{bmatrix},\label{eq:identity_block}
 \end{eqnarray}
 where $O$ denotes the $N\times N$ zero matrix.
 Since $A$ and $B$ commute with each other,
 the upper-right block vanishes.
 The determinant of the identity~\eqref{eq:identity_block}
 yields
 \begin{eqnarray}
  \det \left[M\right] \det \left[A\right] = \det \left[A\right] \det\left[DA-CB\right].
 \end{eqnarray}
 Since $\det\left[A\right]\neq 0$,
 we have 
 \begin{eqnarray}
  \det \left[M\right] = \det\left[DA-CB\right].
 \end{eqnarray}
\end{proof}

\begin{prop}
 Suppose there exists a solution of Eq.~\eqref{eq:nonlinearEQs}
 and let $\bm{F}:\mathbb{R}^{N}\rightarrow \mathbb{R}^{N}$ be Lipschitz
 continuous near the solution~$\bm{x}^{*}$ and
 $X$ and $Y$ be $N\times N$ real nonsingular matrices,
 then the iteration map~\eqref{eq:W4map} with~$Y=X^{-1}J^{-1}$
 and $0<\Delta\tau<1$ yields a series of vectors~$\bm{x}_n$ that converges to the
 solution~$\bm{x}^{*}$
 from an arbitrary initial guess~$\bm{x}_0$ sufficiently close to the solution.
\end{prop}
\begin{proof}
 Let $\lambda_W$ be an eigenvalue of the matrix~$W$.
 Then the following equation is satisfied:
 \begin{eqnarray}
  \det\left[W -\lambda_WI\right] = \det\left[(1-\lambda_W)(1-2\Delta\tau -\lambda_W)I
       +\Delta\tau^2 Y J X \right]=0, \label{eq:eigenW4map0}
 \end{eqnarray}
 where we use Lemma~\ref{lem:blockmatrix}.
 Under the given condition~$Y=X^{-1}J^{-1}$, Eq.~\eqref{eq:eigenW4map0} is reduced to
 \begin{eqnarray}
  \det\left[W -\lambda_WI\right]
   = \det\left[\left( 1 -\Delta\tau -\lambda_W \right)^2 I \right]=0,
 \end{eqnarray}
 leading to $\lambda_W=1-\Delta\tau$.
 The iteration~\eqref{eq:W4map} hence converges to the
 solution~$\bm{x}=\bm{x}^{*}$ and $\bm{p}=0$
 as long as $0<\Delta\tau <1$,
 according to Lemma~\ref{lem:error}.
\end{proof}

The proposition implies the local convergence
 of our formulation. Note that
Eqs.~\eqref{eq:W4equation} are combined into the following second-order differential equations:
\begin{eqnarray}
 \ddot{\bm{x}} +2\dot{\bm{x}} -\dot{X}X^{-1}\dot{\bm{x}} +XY\bm{F}(\bm{x}) =0.
\end{eqnarray}
It is apparent then that 
the third term, which may be regarded as a part of the friction term,
does not vanish unless the matrix~$X$ is constant.
It means that the ``damping factor'' is in general time-dependent,
a key property for the local convergence in our formulation and
completely different from the NR or DN method.

\subsection{UDL decomposition}\label{ssec:UDL}
Next we discuss the choice of the preconditioners~$X$~and~$Y$. 
Although we have several options of the matrices~$X$ and $Y$, we propose
here one of the most useful decompositions.

\begin{defn}
 Suppose the $N\times N$ Jacobian matrix is decomposed into $J:=UDL$
 where $U, D$ and $L$ are $N\times N$ upper triangular, diagonal and lower triangular matrices, respectively.
 In our formulation they are employed to give $X=L^{-1}$ and $Y= D^{-1}U^{-1}$.
\end{defn}

Eqs.~\eqref{eq:W4equation} then become the following:
\begin{subequations}
 \label{eq:W4UL}
  \begin{align}
   \bm{x}_{n+1} =& \bm{x}_{n} + \Delta\tau L^{-1}_n \bm{p}_{n},\\
   \bm{p}_{n+1} =& \left(1 -2\Delta\tau\right)\bm{p}_{n}
   -\Delta \tau D^{-1}_nU^{-1}_n \bm{F}(\bm{x}_{n}),
  \end{align}
\end{subequations}
where $L^{-1}_n, D^{-1}_n$
and $U^{-1}_n$ are the inverted matrices evaluated at $\bm{x}_n$: e.g.,
$L^{-1}_n:= L^{-1}(\bm{x}_{n})$.
As mentioned in Sec.~\ref{ssec:W4}, the eigenvalue in the close vicinity
of the solution is $1-\Delta\tau$ to the first order.

\section{Results}\label{sec:result}
We show here the results of some applications of our new method,
reffered to as the W4 method, to some sample problems and comparisons of
its performance with those of the NR and DN methods.

\subsection{Example 1}\label{ssec:ex1}
We will return to the very first problem for a single variable given in Eq.~\eqref{eq:simple1d}:
\begin{eqnarray*}
 f(x) = \arctan (x) +\sin (x) -1 = 0.
\end{eqnarray*}
\begin{table}[t]
 \begin{center}
  \begin{tabular}{c||c|c|c|c|c|c|c|c|c|c|c|c|c|}
   \diagbox{Method}{$x_i$} & -3.0& -2.5& -2.0& -1.5& -1.0& -0.5& 0& 0.5& 1.0& 1.5& 2.0& 2.5& 3.0\\\hline
   NR & $\infty$& $\infty$& $\infty$& 4& 5& 4& 3& 2& 4& 8& 4& 4& 3\\
   DN & 25& $\infty$& 41& 20& 19& 20& 19& 15& 18& 19& 17& 19& 18\\\hline
   W4 & 1434& 33& 70& 22& 25& 26& 25& 20& 22& 28& 30& 25& 24\\
  \end{tabular}
 \end{center}
 \caption{The same as Table~\ref{tab:comp_ex0_simple1d}, but the
 performance of the W4 method is included.
 The convergence criterion is the same: $|f(x)|<10^{-6}$.
 The time interval is set as $\Delta\tau=0.5$ for the W4 method.
 The symbol~$\infty$ implies again that no solution is obtained within $10^4$ iteration steps.
 }
 \label{tab:W4_ex0_simple1d}
\end{table}
The iteration map for this problem in our W4 method is given as
\begin{subequations}
 \label{eq:W4map_simple1d}
  \begin{align}
   x_{n+1} &= x_{n} + \Delta\tau p_{n},\\
   p_{n+1} &= \left(1-2\Delta\tau\right)p_{n} - \Delta\tau \frac{\arctan(x_n)+\sin(x_n)-1}{\frac{1}{1+x_n^2}+\cos(x_n)},
  \end{align}
\end{subequations}
where $0 < \Delta\tau < 1 $.
In Table~\ref{tab:W4_ex0_simple1d}, we summarize the performance of the
W4 method together with those for the NR and DN methods.
Although the numbers of iteration steps are larger than those
of the DN method,
the W4 method successfully finds the solution where the other two methods fail.
The inertia term or the second derivative term,
which is absent in the NR or DN method, is supposed to be responsible for
the ``slower'' convergence in our method.
The case with $x_i=-2.5$ is an exceptional case, in which this effect is
particularly remarkable.
It is somewhat expected because Theorem~\ref{thm:bendixson} guarantees
that there does not exist any permanent oscillation
in the continuous time-evolution equation.

\subsection{Example 2}\label{ssec:ex2}
Next, we consider the multi-variable problem given in Eqs.~\eqref{eq:simple2d}
\begin{subequations}
 \begin{align*}
  F_1(x,y) &= x^2 +y^2 -4 = 0,\\
  F_2(x,y) &= x^2y -1 = 0.
 \end{align*}
\end{subequations}
The W4 method with the UDL decomposition discussed in Sec.~\ref{ssec:UDL} gives
the following iteration map for this problem:
\begin{subequations}
 \label{eq:FProblem_W4UL}
 \begin{align}
  x_{n+1} &= x_{n} +\Delta\tau p_{n},\\
  y_{n+1} &= y_{n} +\Delta\tau \left[ -\frac{2y_{n}}{x_{n}}p_{n} +q_{n}\right],\\
  p_{n+1} &= \left(1-2\Delta\tau\right)p_{n} -\Delta\tau \left[ \frac{x_{n}}{2\left(x_{n}^2-2y_{n}^2\right)}F_1 -\frac{y_{n}}{x_{n}\left(x_{n}^2-2y_{n}^2\right)}F_2\right],\\
  q_{n+1} &= \left(1-2\Delta\tau\right)q_{n} -\frac{\Delta\tau}{x_{n}^2}F_2,
 \end{align}
\end{subequations}
in which $\bm{x}=\left(x\ y\right)^T$ and $\bm{p}=\left(p\ q\right)^T$.

\begin{figure}[t]
 \begin{tabular}{cc}
  \includegraphics[width=7.cm,clip]{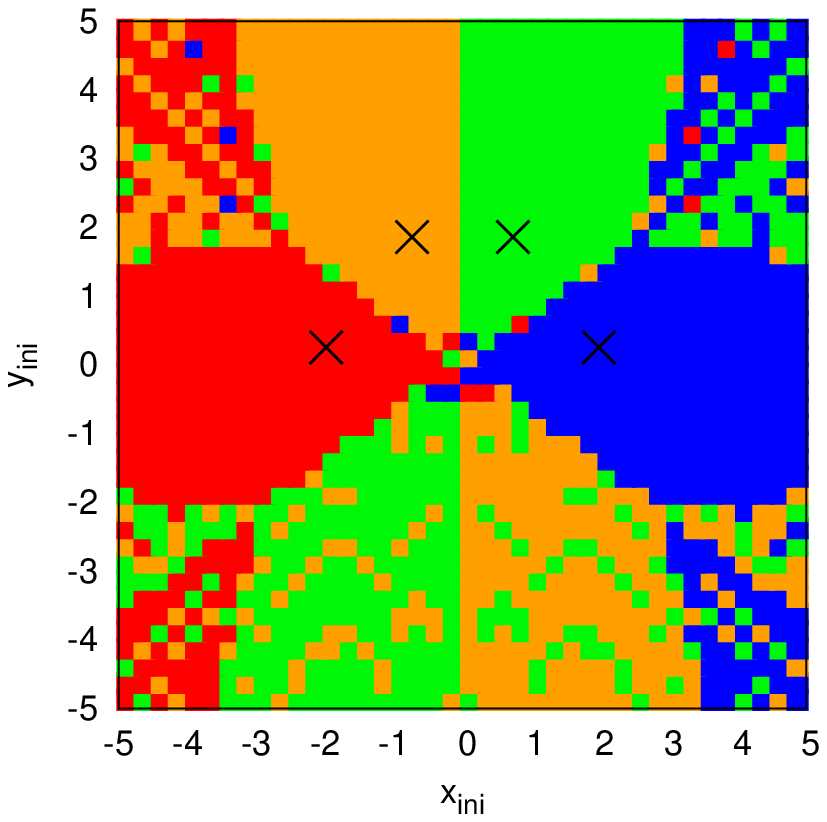}&
  \includegraphics[width=8.8cm,clip]{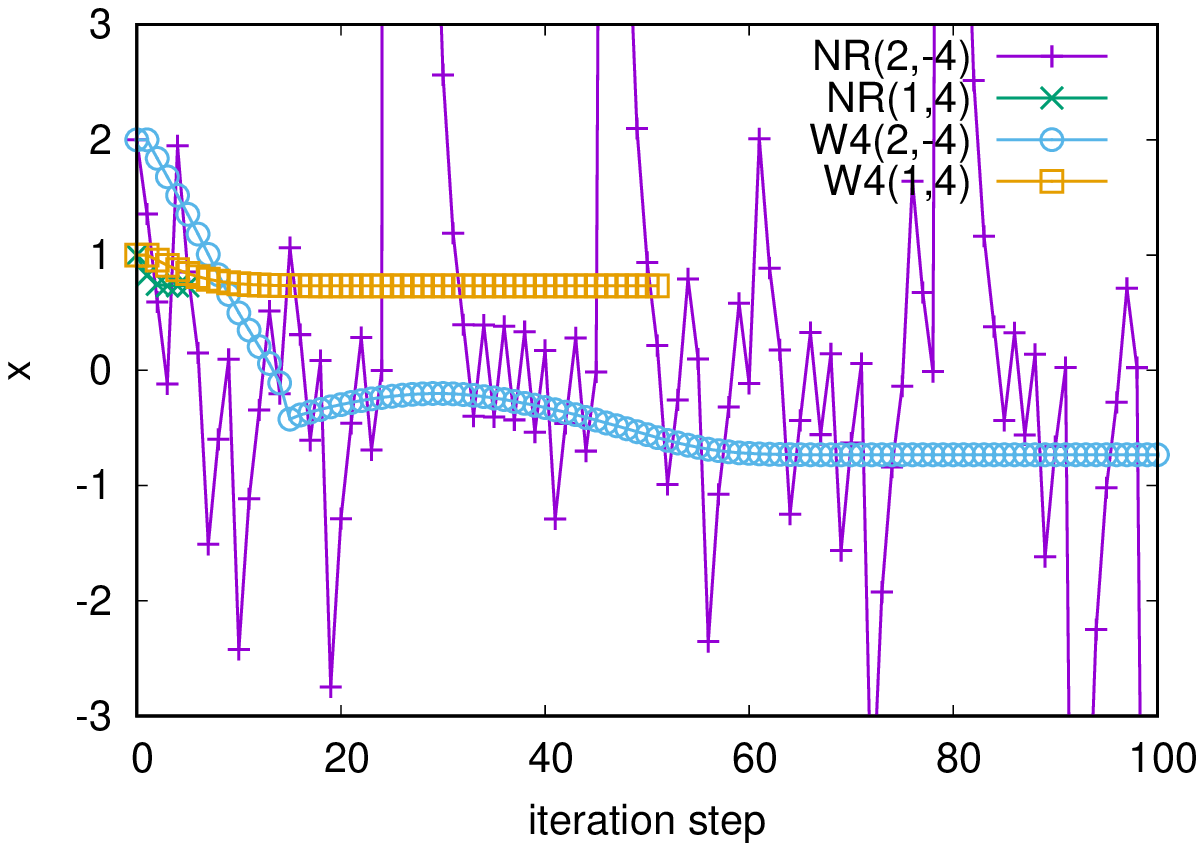}
 \end{tabular}
 \caption{(Left) the Newton basin for the
 W4 method with $X=L^{-1}, Y=D^{-1}U^{-1}$ and $\Delta\tau=0.5$.
 The red, orange, green and blue colors correspond to the initial
 conditions to reach one of the four solutions of Eq.~\eqref{eq:simple2d}
 indicated with the black crosses as in Fig.~\ref{fig:ex0_simple2d_NR}.
 (Right)
 the sequences produced by the itetaion maps of the NR and W4 methods
 for the two initial conditions:
 $(x_0,y_0)=(1,4)$ and $(2,-4)$.
 }
 \label{fig:FProblem_W4}
\end{figure}

In the left panel of Fig.~\ref{fig:FProblem_W4}, we display the
initial-guess dependence of the
W4 method with $X=L^{-1}, Y=D^{-1}U^{-1}$ and $\Delta\tau=0.5$.
Four colors~(red, orange, green, blue) correspond to the four different solutions
just as in~Fig.~\ref{fig:ex0_simple2d_NR}.
We first note that the local convergence is observed also for the W4 method,
 more or less in
a similar way to the NR method as expected.
What is most remarkable, however, is the fact that
the black region in~Fig.~\ref{fig:ex0_simple2d_NR}, where the NR method fails to find a solution within $1000$ iteration steps,
disappears completely in our new method within the initial value space
that we have explored.
In the right panel of Fig.~\ref{fig:FProblem_W4},
we compare the sequences that the iteration maps of the two methods
produce for
the same initial conditions, $(x,y)=(1,4)$ or $(2,-4)$.
It is apparent that the
NR method immediately finds the solution
for the former whereas it fails to obtain a solution for the latter.
In fact the sequence is oscillatory for the latter case and there is no
hint of convergence at least up to $1000$ iterations.
The W4 method, on the other hand,
always manages to find a solution
although it takes more iterations to reach it,
compared with the NR method when it does find a solution.

\subsection{Example 3}\label{ssec:ex3}
The last problem is new and given by the following equations:
\begin{subequations}
 \label{eq:OProblem}
 \begin{align} 
  F_1(x,y) &= x^2 -y^2 -4x +6 = 0,\\
  F_2(x,y) &= 2xy +4y -2 = 0,
 \end{align}
\end{subequations}
which have actually two solutions at
\begin{align}
 \label{eq:OProblem_Solution}
 (x^{*},y^{*})\sim
 ( -1.7505169,  4.0082886),\quad
 ( -2.2244718, -4.4549031).
\end{align}
The Jacobian for this problem is given as
\begin{eqnarray}
 J = 
 \begin{bmatrix}
  2(x-2) & -2y\\
  2y & 2(x+2)
 \end{bmatrix}.
 \label{eq:OProblem_Jacobian}
\end{eqnarray}
\begin{figure}[t]
 \begin{tabular}{cc}
  \includegraphics[width=8.cm,clip]{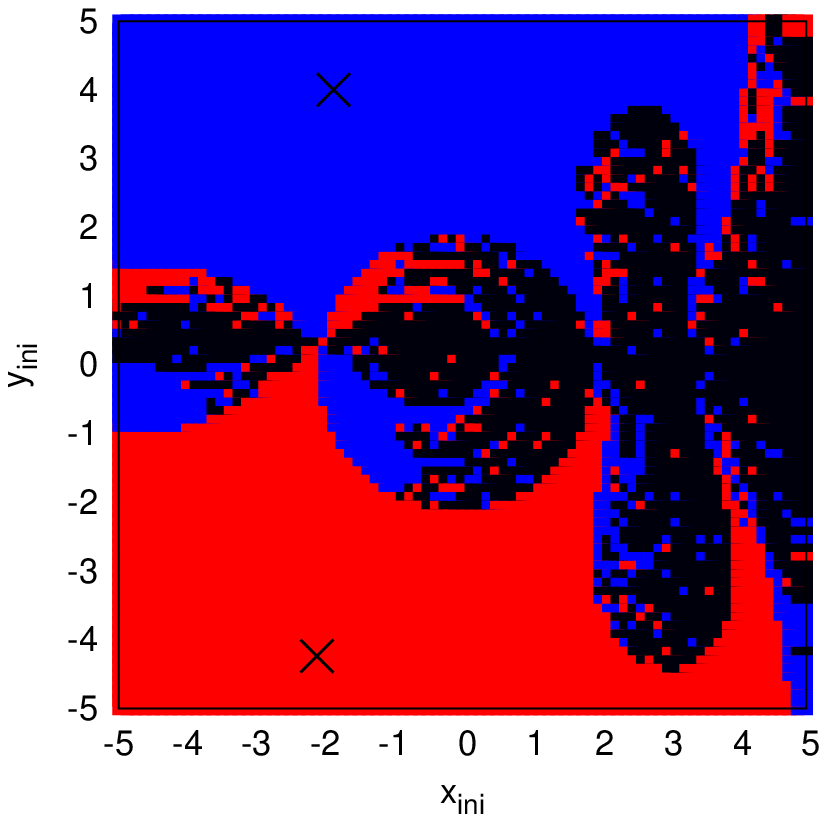}&
  \includegraphics[width=8.cm,clip]{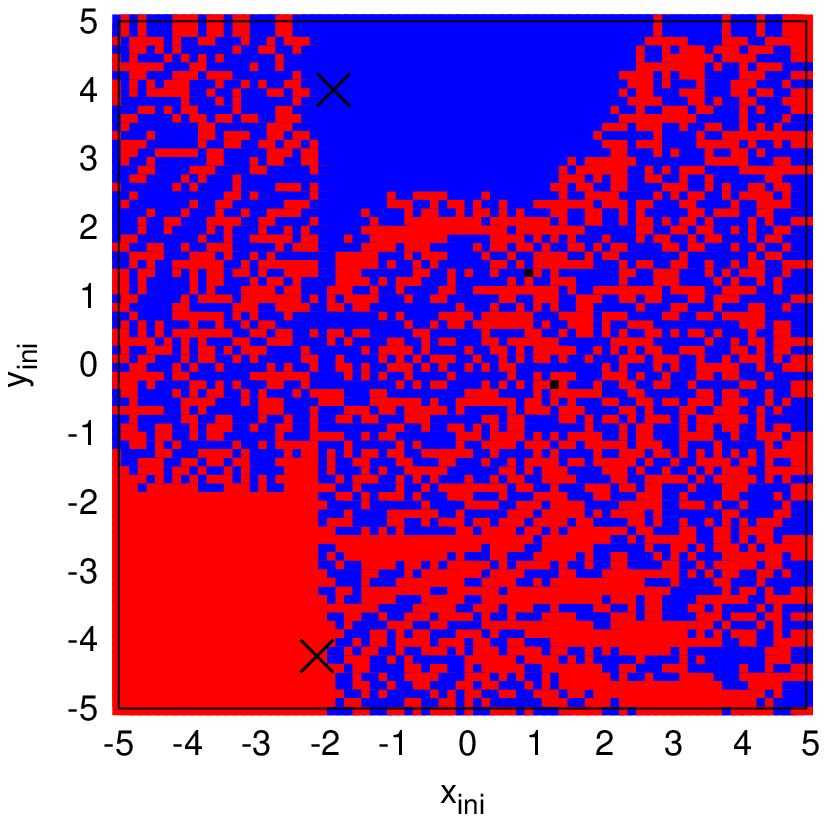}
 \end{tabular}
 \caption{The Newton basins for the NR method~(left) and
 the W4 method with $X=D^{-1}L^{-1}, Y=U^{-1}$ and $\Delta\tau=0.5$~(right).
 The red and blue colors correspond to the initial conditions that
 obtain one of the two solutions of Eq.~\eqref{eq:OProblem} indicated
 with the black crosses,
 while the black color means that no solution is obtained within $1000$ iteration steps.
 }
 \label{fig:OProblem}
\end{figure}

(i) The iteration maps for the NR or DN method are written as
\begin{subequations}
 \label{eq:OProblem_NR}
 \begin{align}
  x_{n+1} &= x_{n} -\frac{\Delta\tau}{2\left(x_{n}^2+y_{n}^2-4\right)}
  \Bigl[\left(x_{n}+2\right)F_1 +y_{n}F_2\Bigr], \\
  y_{n+1} &= y_{n} -\frac{\Delta\tau}{2\left(x_{n}^2+y_{n}^2-4\right)}
  \Bigl[-y_{n}F_1 +\left(x_{n}-2\right)F_2\Bigr],
 \end{align}
\end{subequations}
where $x$ and $y$ are defined as $\bm{x}=\left(x\ y\right)^T$.

(ii) In the W4 method with the UDL decomposition, we solve the following
iteration maps for
$\bm{x}=\left(x\ y\right)^T$ and $\bm{p}=\left(p\ q\right)^T$:
\begin{subequations}
 \label{eq:OProblem_W4}
 \begin{align}
  x_{n+1} &= x_{n} +\Delta\tau p_{n},\\
  y_{n+1} &= y_{n} +\Delta\tau \left[ \frac{y_{n}}{x_{n}+2}p_{n} +q_{n}\right],\\
  p_{n+1} &= \left(1-2\Delta\tau\right)p_{n}
  -\frac{\Delta\tau\left(x_{n}+2\right)}{2\left(x_{n}^2+y_{n}^2-4\right)}
  \left[ F_1 +\frac{y_{n}}{x_{n}+2}F_2\right],\\
  q_{n+1} &= \left(1-2\Delta\tau\right)q_{n} -\frac{\Delta\tau}{2\left(x_{n}+2\right)}F_2.
 \end{align}
\end{subequations}

In Fig.~\ref{fig:OProblem}, we show the Newton basins for the iteration
map~(i) of the NR method~(left) and for the iteration map~(ii) of the
W4 method with $X=L^{-1}, Y=D^{-1}U^{-1}$ and $\Delta\tau=0.5$~(right).
The red and blue colors correspond to the initial conditions that
converge to one of the solutions indicated with the black crosses while
the black dots again imply that no solution is reached within $1000$ iteration steps.
It is found that for the NR method 
the blue or red region is clustered near the solution,
to which it converges and whereas two colors appear rather randomly in
the region between the two solutions.
It is also evident that the NR method sometimes fails to find a solution
whereas the W4 method always reaches one of the solutions.
Which one is obtained is unpredictable.

%%%%%%%%%%%%%%%%%%%%%%%%%%%%%%%%%%%%%%%%%%%%%%%%%%%%%%%%%%%%%%%%%%%%%%%%%%%%%%%%%
\subsection{Global convergence}\label{sec:discussion}
%%%%%%%%%%%%%%%%%%%%%%%%%%%%%%%%%%%%%%%%%%%%%%%%%%%%%%%%%%%%%%%%%%%%%%%%%%%%%%%%%
%
In this section, we elucidate the difference between the NR method and
the W4 method, using the following
two-variable problem, which is somewhat similar to Example~2 in Sec.~\ref{ssec:ex2},
\begin{subequations}
 \label{eq:FProblem0}
 \begin{align}
  F_1(x,y) &= x^2 +xy^2 -4 = 0,\\
  F_2(x,y) &= x^2y -1 = 0,
 \end{align}
\end{subequations}
which has three solutions at
\begin{align}
 \label{eq:FProblem0_Solution}
 (x^{*},y^{*})\sim
 (-2.0296789, 0.24274223),\quad
 ( 1.9668697, 0.25849302),\quad
 ( 0.65417501, 2.3367492).
\end{align}
The Jacobian matrix for this system is
given as 
\begin{eqnarray}
 J = 
 \begin{bmatrix}
  2x +y^2 & 2xy\\
  2xy & x^2
 \end{bmatrix}.\label{eq:FProblem0_Jacobian}
\end{eqnarray}
Note that this is a $2\times 2$ symmetric matrix and
has two real eigenvalues~$\lambda_{\pm}$:
\begin{eqnarray}
 \lambda^{\pm} &=& \frac{1}{2}
  \left( A \pm \sqrt{A^2 -4B}\right),\label{eq:FProblem0_eigenvalue}
\end{eqnarray}
where we define $A:= {\rm tr}J$ and $B:= \det J$.
Then the corresponding eigenvectors are given as
\begin{eqnarray}
 \bm{v}^{\pm} = \left[ \frac{2xy}{a^{\pm}}\ \frac{\lambda_{\pm}-2x-y^2}{a^{\pm}}\right]^{T},\quad
 a^{\pm} := \sqrt{4x^2y^2 +\left(\lambda^{\pm} -2x-y^2 \right)^2}.%\nonumber\\
 \label{eq:FProblem0_eigenvector}
\end{eqnarray}
\begin{figure}[t]
 \begin{center}
  \begin{tabular}{cc}
   \includegraphics[width=8.4cm,clip]{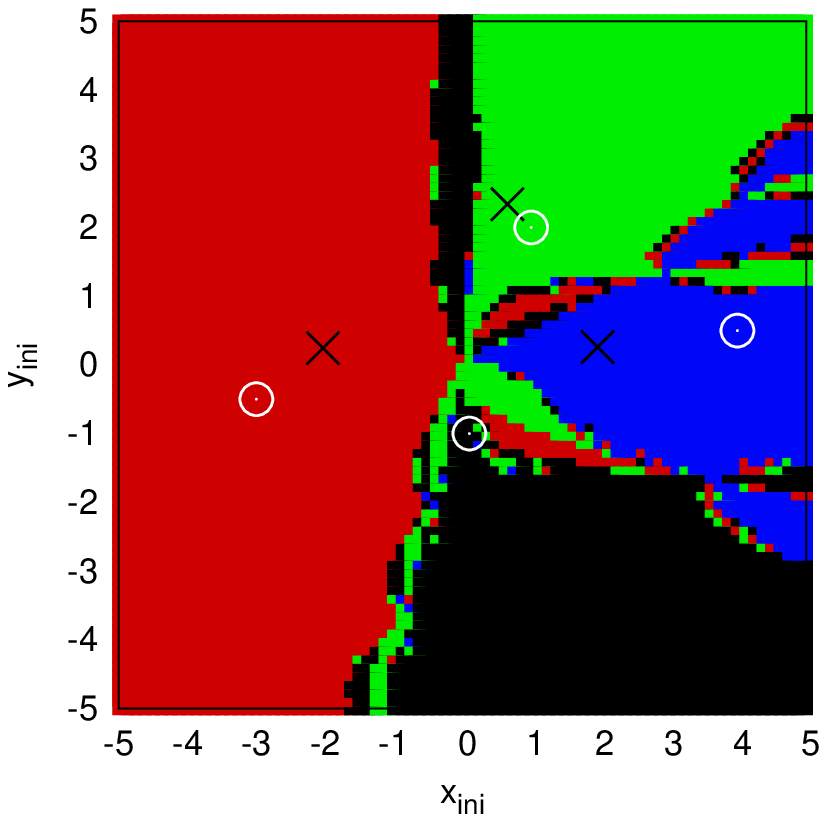} &
   \includegraphics[width=8.4cm,clip]{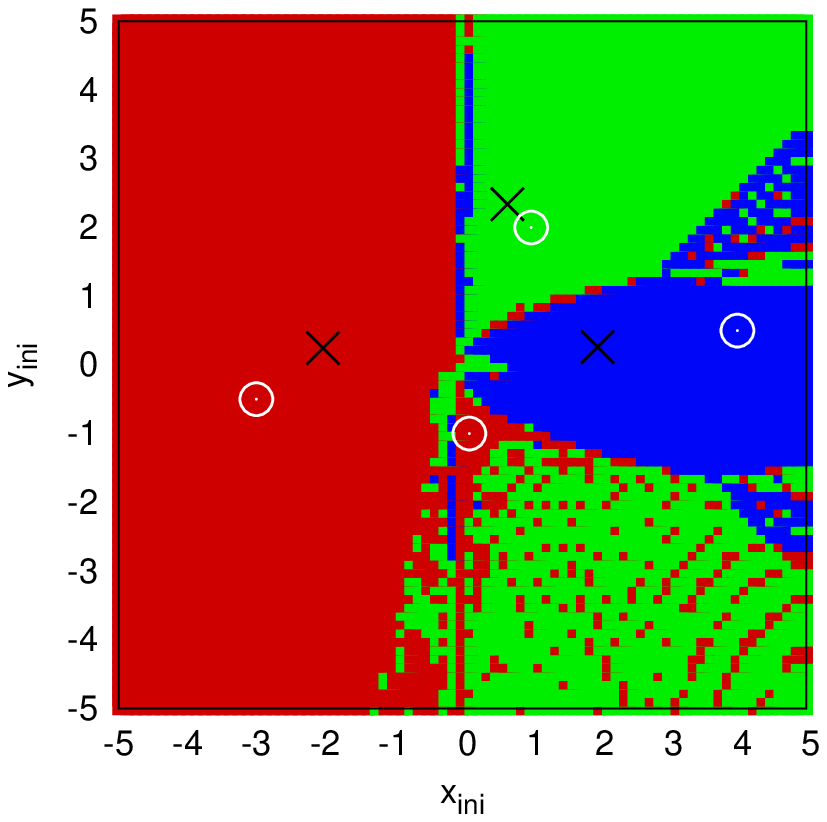}\\
   \includegraphics[width=7.8cm,clip]{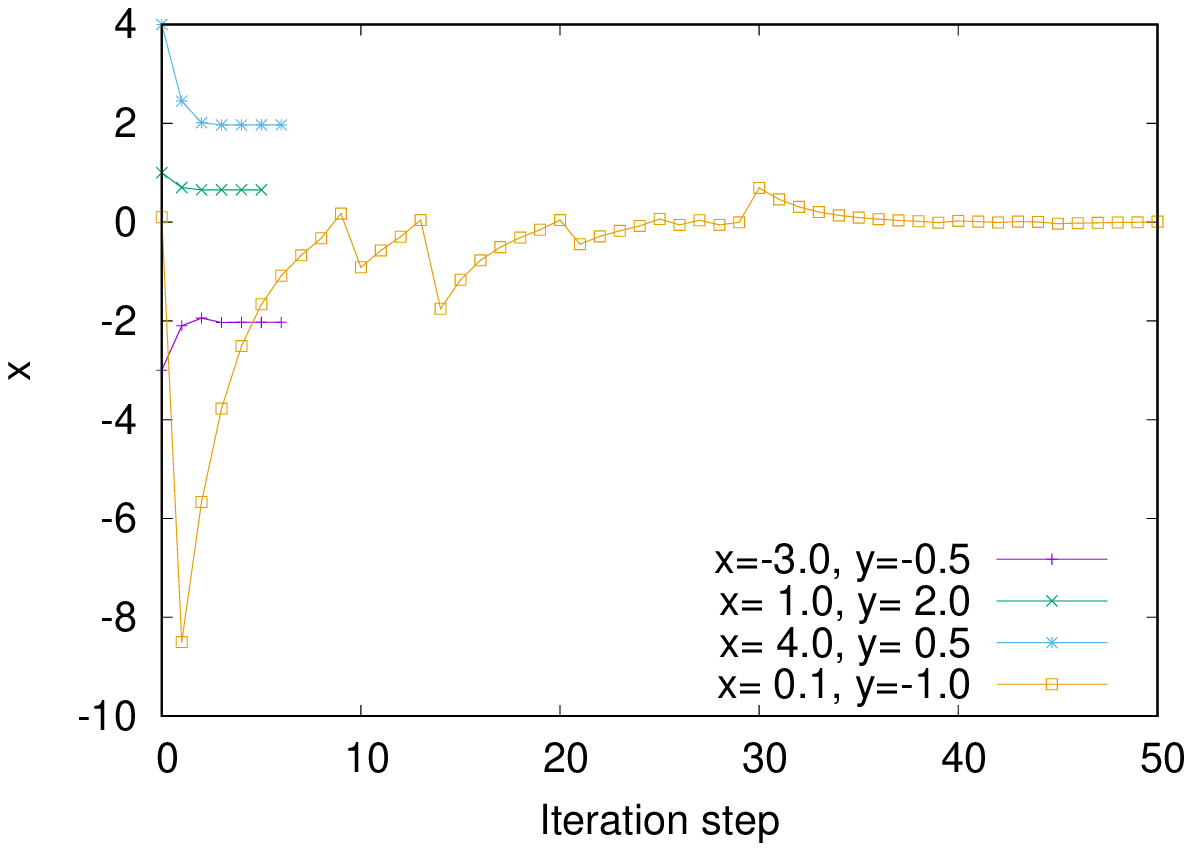} &
   \includegraphics[width=7.8cm,clip]{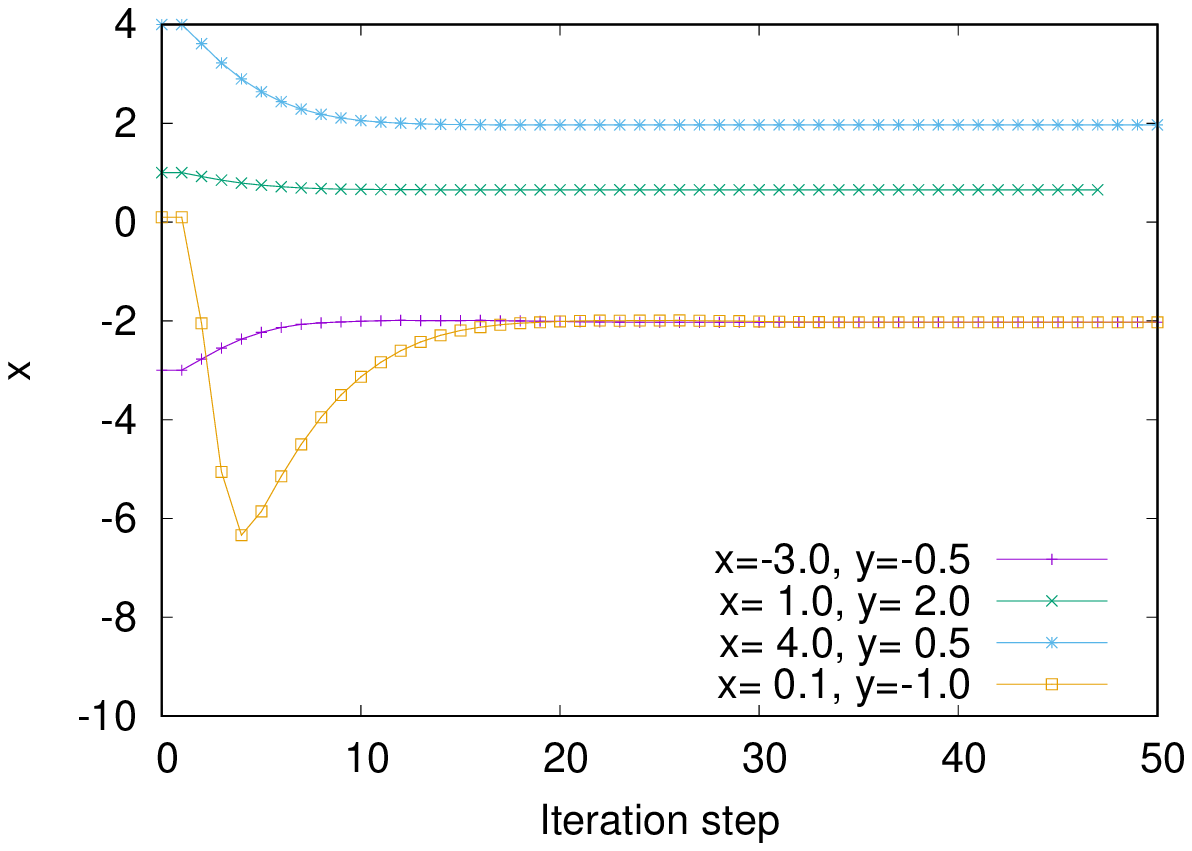}
  \end{tabular}
 \end{center}
 \caption{Initial guess dependence by NR method(Upper-Left)
 and W4 method with $X=P, Y=\Lambda^{-1} P^{-1}$ and  $\Delta\tau=0.5$(Upper-Right).
 The red, green and blue correspond to all solutions of Eq.~\eqref{eq:FProblem0},
 while the black means the method cannot find any solution before $1000$ iterations.
 Each solution is described by the black crosses.
 The white circles are initial guesses for the lower figures showing its
 evolutions by NR method(Lower-Left) and our new method(Lower-Right).
 }
 \label{fig:FProblem0}
\end{figure}
\begin{figure}[t]
 \begin{tabular}{cc}
  \includegraphics[width=7.8cm,clip]{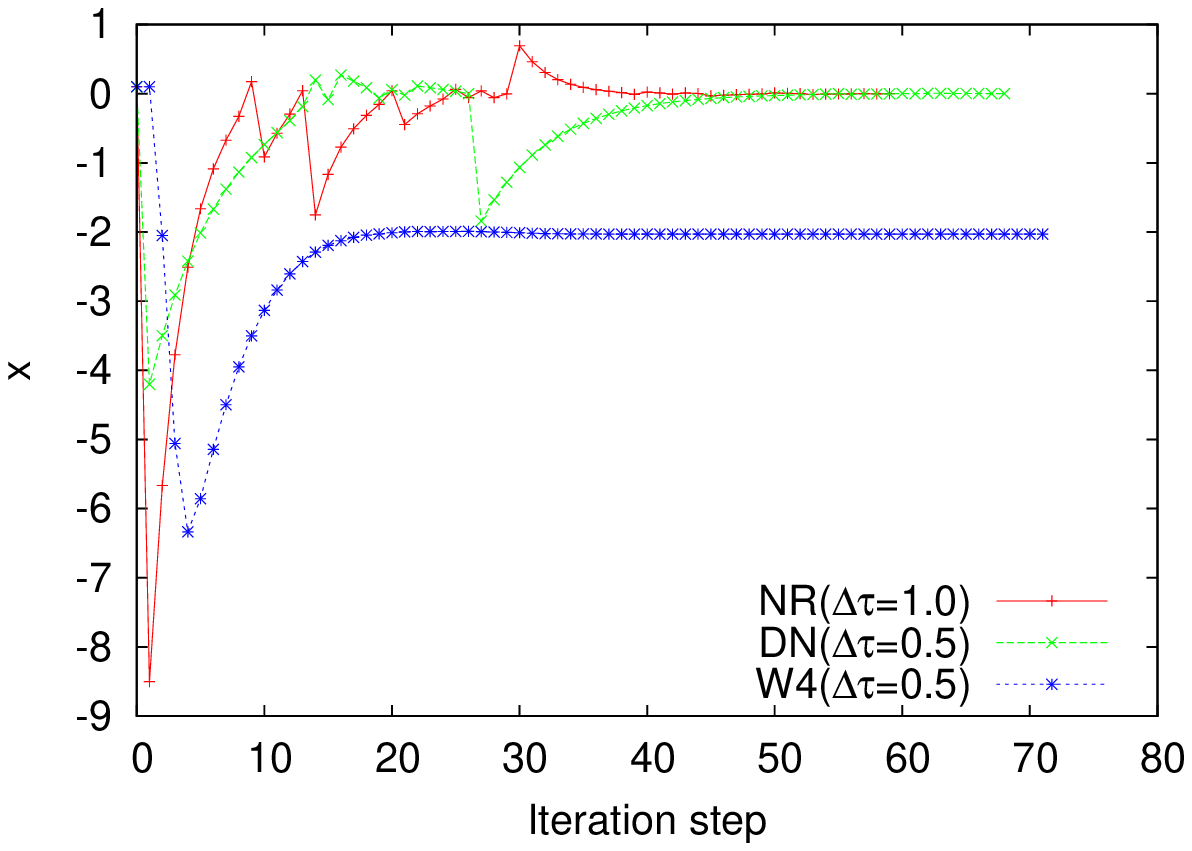} &
  \includegraphics[width=7.8cm,clip]{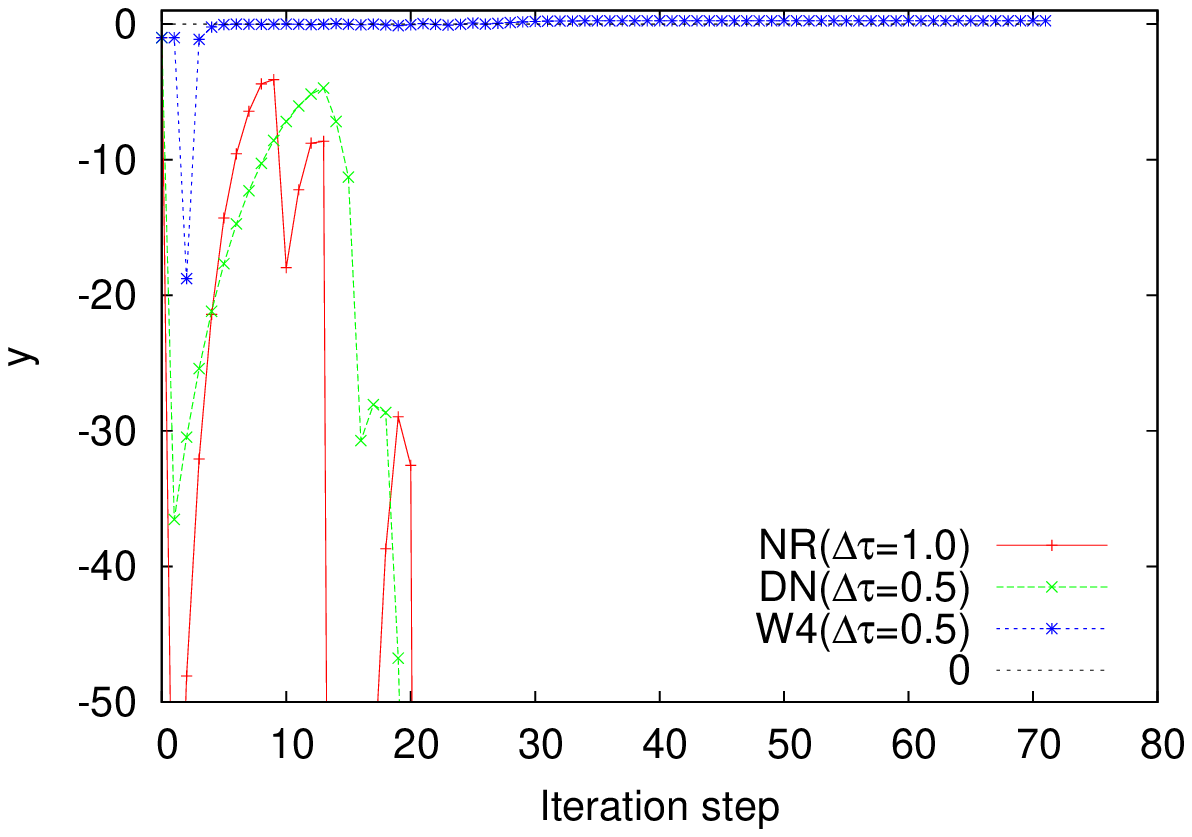}\\
  \includegraphics[width=7.8cm,clip]{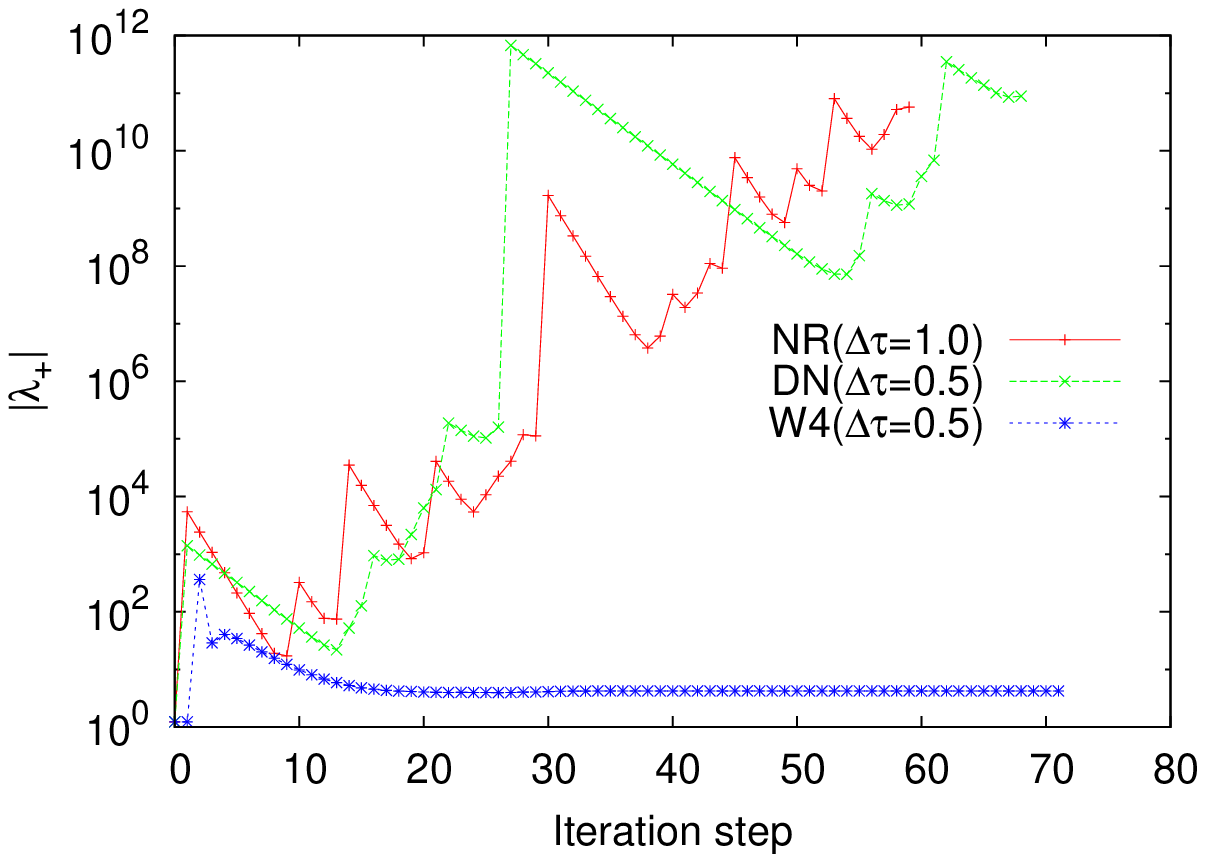} &
  \includegraphics[width=7.8cm,clip]{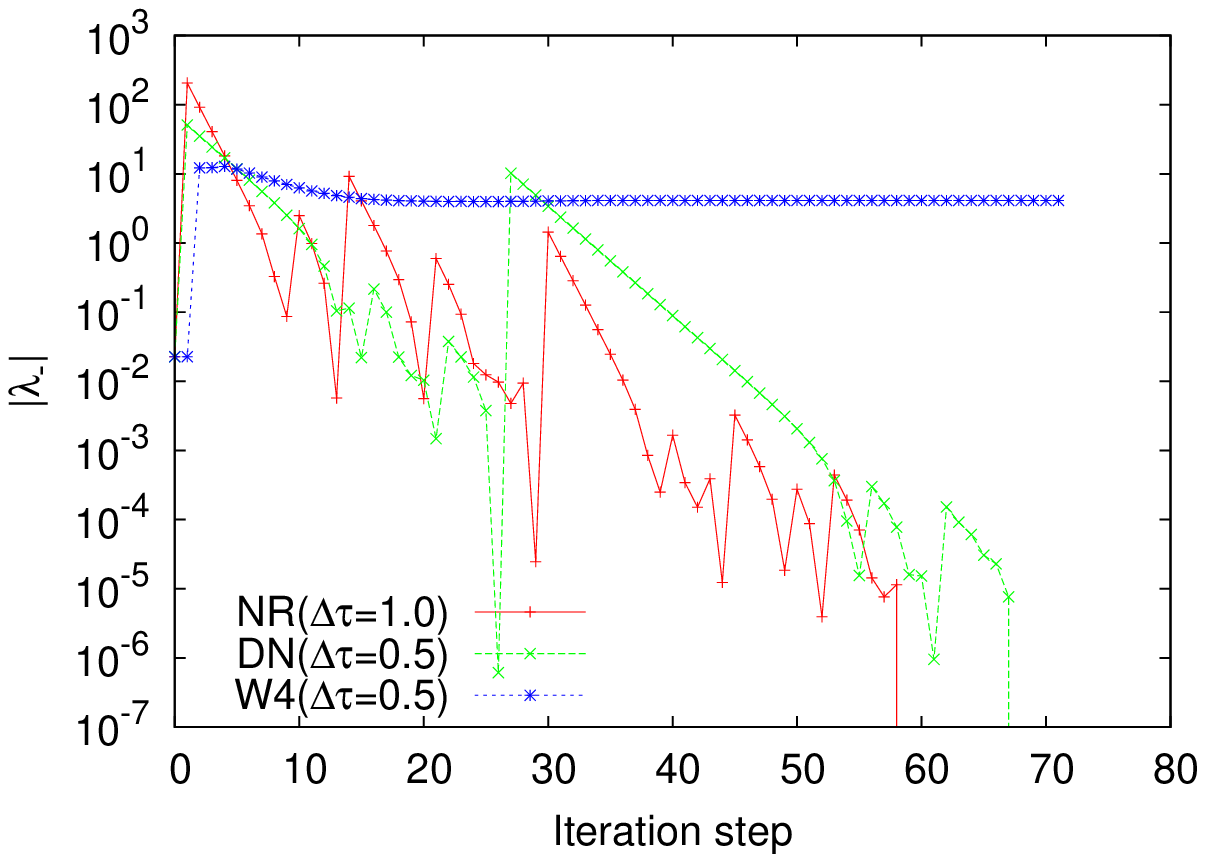}
 \end{tabular}
 \caption{The evolutionary sequences of $x$~(upper left),
 $y$~(upper right), $|\lambda_{+}|$~(lower left)
 and $|\lambda_{-}|$~(lower right) generated by
 the iteration maps of the NR, DN and W4 methods.
 $\lambda_{\pm}$ are the maximum and minimum eigenvalueas
 of the Jacobian matrix, respectively.
 }
 \label{fig:FProblem0_detail}
\end{figure}

The iteraion map for the DN method can be expressed in terms of these
eigenvalues and eigenvectors as
\begin{align}
 \label{eq:NRPP}
 \bm{x}_{n+1} =& \bm{x}_{n} - P_n\Lambda^{-1}_nP^{-1}_n \bm{F}(\bm{x}_n)\Delta\tau,
\end{align}
where we introduce
$P_n=Q_n/\det(Q_n),\ Q_n:=\left[\bm{v}^{+}_n\ \bm{v}^{-}_n\right]$,
 and
$\Lambda^{-1}_n={\rm diag}\left[\frac{1}{\lambda^{+}_n}\ \frac{1}{\lambda^{-}_n}\right]$ for nonzero $\lambda^{-}_{n}$
and $\Lambda^{-1}_n={\rm diag}\left[\frac{1}{\lambda^{+}_n}\ 0\right]$ for $\lambda^{-}_{n}=0$.
Note that the determinant of $P_n$ is unity.
On the other hand, the iteration map for the W4 map is given as
\begin{subequations}
 \label{eq:W4PP}
  \begin{align}
   \bm{x}_{n+1} =& \bm{x}_{n} + P_n \bm{p}_{n}\Delta\tau,\\
   \bm{p}_{n+1} =& \left(1 -2\Delta\tau\right)\bm{p}_{n}
   -\Lambda^{-1}_n P^{-1}_n \bm{F}(\bm{x}_n)\Delta\tau.
  \end{align}
\end{subequations}
Note in particular that the eigenvalues and eigenvectors are
defined pointwise and change for each iteration
in general.

In the upper panels of Fig.~\ref{fig:FProblem0}, we exhibit the Newton
basins for the NR~(left) and the W4~(right) methods.
The red, blue, and green colors correspond to the
initial conditions that produce one of the three solutions
indicated with the black crosses whereas
the black points indicate that no solution is reached within $1000$ iterations.
It is clear that the 
W4 method always finds some solution in sharp contrast to the NR method,
which sometimes fails to reach any of the solutions.
In the lower panels of Fig.~\ref{fig:FProblem0},
we show the specific sequences starting from the white circles
displayed in the upper panels.
The NR method reaches a solution very quickly
within $10$ iterations
normally whenever it is successful.
This happens usually when the initial guess is close to the solution.
On the other hand,
the W4 method needs more iterations
but always manages to find some solution.
In the following, we investigate the possible cause of this difference more
in detail, picking up the case for the initial guess~$(x_i,y_i)=(0.1,-1.0)$.\\

We display
in the upper panels of Fig.\ref{fig:FProblem0_detail}
three evolutions of~$x$~(left) and $y$~(right) by the NR method, the DN
method with $\Delta\tau=0.5$, and the W4 method with $\Delta\tau=0.5$.
One can see that $(x_n,y_n)$ go to $(0,-\infty)$ as $n\rightarrow \infty$
both for the NR method and DN methods.
In the lower panels of Fig.\ref{fig:FProblem0_detail},
we draw the evolutions of absolute values of the maximum~$(\lambda^{+}, left)$
and minimum $(\lambda^{-}, right)$ eigenvalues.
It is evident that the ratio of $|\lambda^{-}/\lambda^{+}|$
tends to zero again both in the NR and the DN method.
In what follows, we pay particular attention to this behavior
of $\lambda^{-}/\lambda^{+}$ and analyze
the situation when $|\lambda^{-}/\lambda^{+}|$
becomes quite small at $x\sim 0$.

\begin{lem}\label{lem:NRmapPP}
 Suppose there exists a complete set of eigenvectors~$\bm{v}^{\pm}_n\in \mathbb{R}^{2}$
 of the symmetric Jacobian matrix~$J_n$ at the $n$-th iteration step
and let $P_n$ be an $N\times N$ orthogonal matrix composed
 of~$\bm{v}^{\pm}_n$ and $\lambda^{\pm}_n$ be
 the corresponding eigenvalues,
 then the increment~$\bm{x}_{n+1}-\bm{x}_{n}$ at this
 particular step tends to be aligned with $\bm{v}^{-}_{n}$
 if $|\lambda^{-}_n/\lambda^{+}_n|\rightarrow 0$ as $n\rightarrow\infty$.
\end{lem}

\begin{proof}
 The iteration map~\eqref{eq:NRPP} with $\Delta\tau=1$ gives
 the increment in~$\bm{x}$ at the $n$-th step
 in terms of the eigenvectors as
\begin{eqnarray}
 \bm{x}_{n+1} -\bm{x}_{n} = -P_n\Lambda^{-1}_{n} P^{-1}_{n} \bm{F}(\bm{x}_{n})
  = -\frac{c^{+}_n}{\lambda^{+}_n}\bm{v}_n^{+} -\frac{c^{-}_n}{\lambda^{-}_n}\bm{v}_n^{-}
  =  -\frac{\lambda^{-}_n}{\lambda^{+}_n}\frac{c^{+}_n}{\lambda^{-}_n}\bm{v}_n^{+} -\frac{c^{-}_n}{\lambda^{-}_n}\bm{v}_n^{-}
  \rightarrow  -\frac{c^{-}_n}{\lambda^{-}_n}\bm{v}_n^{-},
\end{eqnarray}
where we introduce the inner product
 $c^{\pm}_n:=(\bm{v}^{\pm}_n)^{T}\bm{F}(\bm{x}_n)$ and
 take the limit~$|\lambda^{-}_n/\lambda^{+}_n|\rightarrow 0$.
\end{proof}
This lemma explains the behavior of the sequences
produced by the iteration map of the DN method.
In fact, Taylor-expanding
the eigenvalues and eigenvectors
as well as the coefficients~$c^{\pm}$ in the neighborhood
of $x=0$,
\begin{eqnarray}
 \lambda^{+} = y^2 +2x +4x^2 +\mathcal{O}\left(x^3\right),\quad
 \lambda^{-} = -3x^2 +\mathcal{O}\left(x^3\right),\\
 \bm{v}^{+} = \left(-1 +\frac{2x^2}{y^2} +\mathcal{O}\left(x^3\right)\quad
	      -\frac{2x}{y} +\frac{4x^2}{y^2}+\mathcal{O}\left(x^3\right)\right)^{T},\\
 \bm{v}^{-} = \left(
	      \frac{2x}{y}
	      -\frac{4x^2}{y^2}+\mathcal{O}\left(x^3\right)\quad
	       -1 +\frac{2x^2}{y^2} +\mathcal{O}\left(x^3\right)
			    \right)^{T},\\
 c^{+} = 4 +\left(-y^2 +\frac{2}{y}\right)x +\mathcal{O}\left(x^2\right),\quad
  c^{-} = 1 -\frac{8x}{y} +\mathcal{O}\left(x^2\right),
\end{eqnarray}
one realizes that once the sequence enters the region with~$y<0$,
it heads to $x=0$ and $y=-\infty$.
This is also understood from the fact that
$\bm{v}^{+}$ is the only way out of this region.

\begin{lem}\label{lem:W4mapPP}
 Suppose there exists a complete set of eigenvectors~$\bm{v}^{\pm}_{n-1}\in \mathbb{R}^{2}$
 of the symmetric Jacobian matrix~$J_{n-1}$ also at the $(n-1)$-th step and
 let $P_{n-1}$ be an $N\times N$ orthogonal matrix composed
 of~$\bm{v}^{\pm}_{n-1}$ and $\lambda^{\pm}_{n-1}$ be the corresponding eigenvalues,
 then in the iteration map of the W4 method with~$\Delta\tau=0.5$
 the increment in~$\bm{x}$ is not aligned completely
 with $\bm{v}^{-}_{n-1}$ but depends also on $\bm{v}^{+}_{n-1}$
 as
 $|\lambda^{-}_{n-1}/\lambda^{+}_{n-1}|\rightarrow 0$
 if $\bm{v}^{+}_{n}\neq\bm{v}^{+}_{n-1}$.
\end{lem}
\begin{proof}
 From the iteration map~\eqref{eq:W4PP} of the W4 method
 with $\Delta\tau=0.5$,
 we can obtain 
\begin{eqnarray}
 \bm{x}_{n+1} -\bm{x}_{n} = -\frac{1}{4}P_n\Lambda^{-1}_{n-1} P^{-1}_{n-1} \bm{F}(\bm{x}_{n-1})
  = -\frac{1}{4}\frac{\lambda^{-}_{n-1}}{\lambda^{+}_{n-1}}\frac{c^{+}_{n-1}}{\lambda^{-}_{n-1}}\bm{v}_n^{+} 
  -\frac{1}{4}\frac{c^{-}_{n-1}}{\lambda^{-}_{n-1}}\bm{v}_{n}^{-}
  \rightarrow   -\frac{1}{4}\frac{c^{-}_{n-1}}{\lambda^{-}_{n-1}}\bm{v}_{n}^{-}.
\end{eqnarray}
 We find the same alignment of the increment with the
 eigenvector at the $n$-th iteration step.
 Note, however, that the degeneracy of the Jacobian, or
 $|\lambda^{-}_{n-1}/\lambda^{+}_{n-1}|\approx 0$,
 is assumed to occur at the $(n-1)$-th iteration step this time.
 The above equation implies that the alignment occurs
 one step later.
Since we can express
 $\bm{v}^{-}_{n}$ in terms of the eigenvector at the $(n-1)$-th
iteration step as
 \begin{eqnarray}
 \bm{v}^{-}_{n} = a^{+}\bm{v}^{+}_{n-1} +a^{-}\bm{v}^{-}_{n-1},
 \end{eqnarray}
where $a^{\pm}$ are certain constants,
the increment in $\bm{x}_{n}$ can be rewritten as
\begin{eqnarray}
 \bm{x}_{n+1} -\bm{x}_{n} =
  -\frac{a^{+}c^{-}_{n-1}}{4\lambda^{-}_{n-1}}\bm{v}_{n-1}^{+}
  -\frac{a^{-}c^{-}_{n-1}}{4\lambda^{-}_{n-1}}\bm{v}_{n-1}^{-},
\end{eqnarray}
 and hence includes the contribution from $\bm{v}^{+}_{n-1}$
unless $a^{+}$ is exactly zero, $a^{+}=0$.
\end{proof}
The important thing in the above lemma is that
the degeneracy of the Jacobian matrix at a certain iteration step
 affects the direction of the sequence at the next step
 but $\bm{v}^{+}_{n-1}$ is the direction of crucial importance
 so that the sequence could get out of the vicinity of $\bm{x}=0$,
in which the sequence of the DN method is trapped.

\begin{figure}[t]
 \begin{tabular}{cc}
  \includegraphics[width=7.8cm,clip]{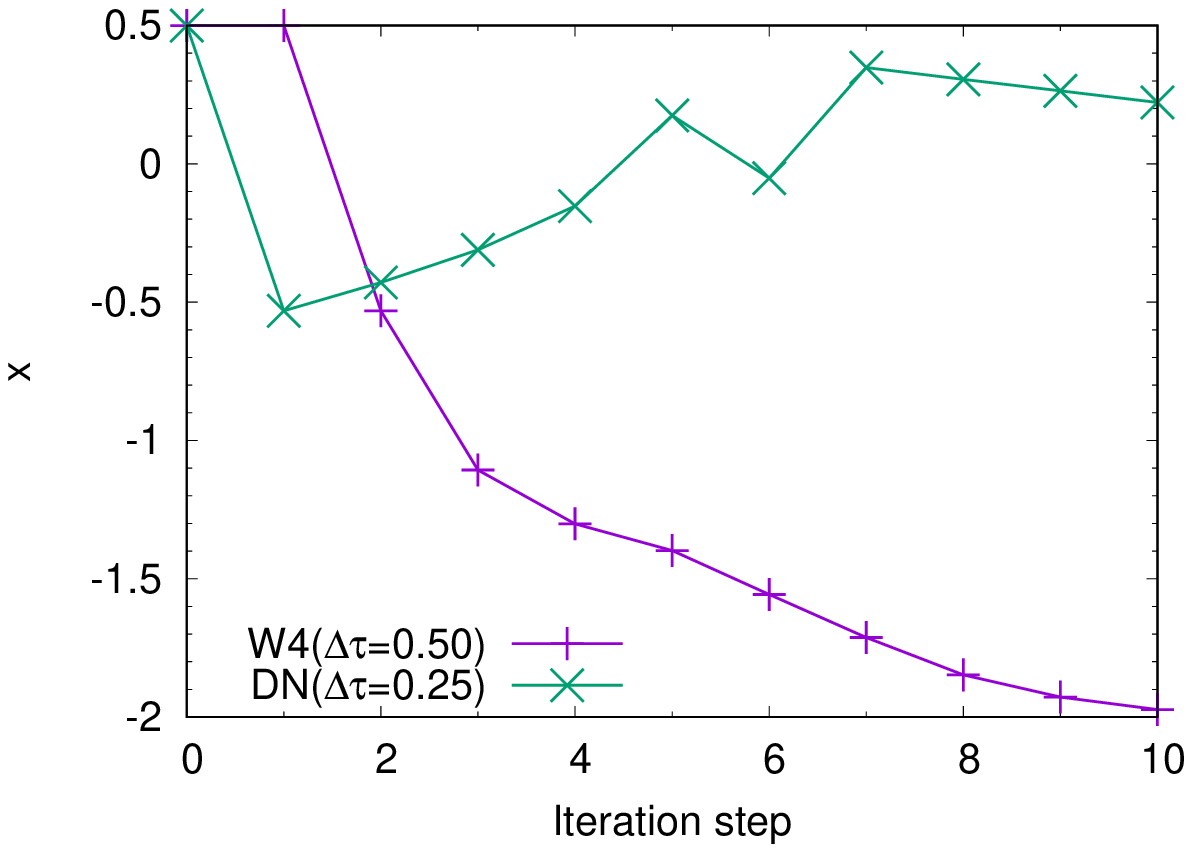} &
  \includegraphics[width=7.8cm,clip]{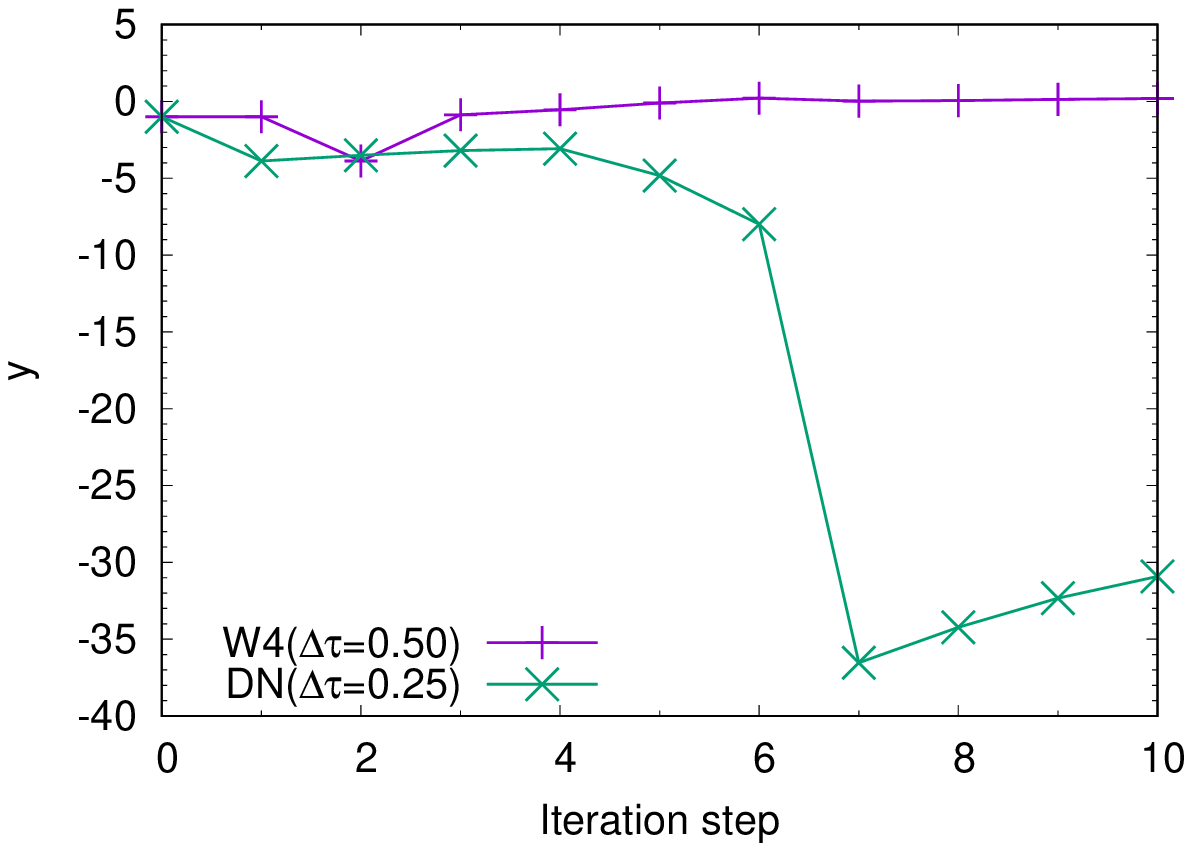}
 \end{tabular}
 \caption{The time evolutions of $x$~(left) and $y$~(right)
 for the DN method with $\Delta\tau=0.25$
 and the W4 method with $\Delta\tau=0.5$.
 }
 \label{fig:FProblem0_inertia}
\end{figure}
Fig.~\ref{fig:FProblem0_inertia} shows the specific evolutions of $x$~(left)
and $y$~(right) for the iteration maps of the DN method with $\Delta\tau=0.25$
and of the W4 method with $\Delta\tau=0.5$,
both starting from the same initial condition~$(x,y)=(0.5,-1.0)$
and $(p,q)=(0,0)$.
Note that $(x_1,y_1)$ at the first step in the DN method is
identical to $(x_2,y_2)$ in the W4 method
but the following evolutions are different from each other.
In fact,
$x_2$ in the DN method gets closer to $x=0$
whereas $x_3$ in the W4 method is farther away,
since the momentum gained in the previous steps
prevents it from changing the direction of motion
in the latter method.

%%%%%%%%%%%%%%%%%%%%%%%%%%%%%%%%%%%%%%%%%%%%%%%%%%%%%%%%%%%%%%%%%%%%
\section{Conclusion}\label{sec:conclusion}
%%%%%%%%%%%%%%%%%%%%%%%%%%%%%%%%%%%%%%%%%%%%%%%%%%%%%%%%%%%%%%%%%%%%
%
We have proposed a new iterative method, which we named the W4 method, to solve nonlinear systems of equations.
In this paper, we have considered a couple of representative problems,
for which the Newton-Raphson method, or the NR method,
fails to find any one of the solutions and have applied the new method
to them to see if any improvement is made.
The results are summarized in Table~\ref{tab:W4_ex0_simple1d} for
the single-variable problems
and in Fig.~\ref{fig:ex0_simple2d_NR} and~\ref{fig:FProblem_W4} for the multi-variable problems.\\

Our findings are listed as follows:
 \begin{enumerate}
  \item When the NR method or its modified version, the DN method, fails
	in these problems, their iteration maps produce sequences, which
	either diverge or oscillate as shown in Fig.~\ref{fig:FProblem_W4}.
  \item The W4 method, which
	introduces the second time derivative term,
	can be hence regarded as an extension of the NR method as argued
	in Sec.~\ref{ssec:W41D} and Sec.~\ref{ssec:W4}.
  \item The W4 method has the same local convergence 
	as the NR method is well-known to possess, as demonstrated 
	in Sec.~\ref{ssec:W4}.
  \item In the W4 method the Newton basins that lead to no solution in the NR method
	are eliminated as exhibited in Figs.~\ref{fig:ex0_simple2d_NR},  \ref{fig:FProblem_W4} and \ref{fig:OProblem}.
	This strongly suggests that the W4 method has a
	global-convergence property in
	multi-variable problems.
  \item In the W4 method, the convergence occurs only linearly just as in
	the DN method.
	The NR method, in contrast, has a quadratic convergence nature.
	This was discussed in
	Sec.~\ref{ssec:NR} and Sec.~\ref{ssec:W4}.
  \item The computational cost of the current version of the W4 method,
	which employs the UDL decomposition of the Jacobian matrix, is
	no higher than that of the NR method because of the
	decomposition of Jacobian into $U,D$ and $L$.
 \end{enumerate}

We believe that our new method, having the better global-convergence
nature, should be very useful in many fields of science, particularly
when the NR method fails to work more often than not.

%%%%%%%%%%%%%%%%%%%%%%%%%%%%%%%%%%%%%%%%%%%%%%%%%%%%%%%%%%%%%%%%%%%%
\section{Acknowledgements}
%%%%%%%%%%%%%%%%%%%%%%%%%%%%%%%%%%%%%%%%%%%%%%%%%%%%%%%%%%%%%%%%%%%%
We would like to thank K.-i Maeda for helpful comments.  We are grateful
to Y.~Eriguchi for valuable comments on the preliminary draft.  This
work was supported by JSPS KAKENHI Grant Number 16K17708, 16H03986,
17K18792.  R.~H. was supported by JSPS overseas research fellowship
No.29-514.  K.~F was supported by JSPS postdoctoral fellowships
No.16J10223.

%%%%%%%%%%%%%%%%%%%%%%%%%%%%%%%%%%%%%%%%%%%%%%%%%%%%%%%%%%%%%%%%%%%%%%%%%%%%%%%%%%%%%%%%%%%
\bibliographystyle{elsarticle-num}
%\bibliography{nonlinear}

\end{document}